\documentclass{article}
\usepackage{amssymb}
\usepackage{amsmath}
\usepackage{amsthm}
\usepackage{amstext}
\usepackage{amsfonts}
\usepackage{authblk}

 \newtheorem{thm}{Theorem}[section]
 \newtheorem{cor}[thm]{Corollary}
 \newtheorem{lem}[thm]{Lemma}
 \newtheorem{prop}[thm]{Proposition}
 \theoremstyle{definition}
 \newtheorem{defn}[thm]{Definition}
 \theoremstyle{remark}
 \newtheorem{rem}[thm]{Remark}
 
 \numberwithin{equation}{section}
\newcommand{\be}{\begin{equation}}
\newcommand{\ee}{\end{equation}}
\def\ds{\displaystyle}
\newcommand{\overbar}[1]{\mkern 1.5mu\overline{\mkern-1.5mu#1\mkern-1.5mu}\mkern 1.5mu}

\title{Theoretical study of  nonlinear  wave equations with combined power-type nonlinearities with  variable coefficients}
\date{}

\author[1,2]{M. Dimova\thanks{Corresponding author: M. Dimova, mdimova@unwe.bg}}
\author[2]{N. Kolkovska}
\author[2]{N. Kutev}

\affil[1]
{University of National and World Economy, Sofia, Bulgaria}

\affil[2]{
Institute of Mathematics and Informatics,
 Bulgarian Academy of Sciences, Sofia, Bulgaria}

\begin{document}
\maketitle

\begin{abstract}
In this paper, we study the initial boundary value problem for the nonlinear wave equation with
combined power-type nonlinearities with  variable coefficients. 
The global behavior of the solutions with non-positive and sub-critical energy  is completely investigated.  
The threshold  between the nonexistence of global in time weak solutions and non-blowing up solutions is found. 
For super-critical energy, two 
 new sufficient conditions guaranteeing  nonexistence of global in time solutions are given. One of them is proved for arbitrary sign of the scalar product of the initial data, while the other 
one is derived only for  positive sign.
Uniqueness and existence of local weak solutions are proved.


\end{abstract}

\maketitle

\section{Introduction}\label{sec1} 
In this paper we consider the  following initial boundary value problem for the nonlinear wave equation  
\begin{align}
&u_{tt} -  \Delta u = f(x,u),  \quad  t>0,  \ \ x \in\Omega, \label{1}\\
&u(0,x)=u_0(x),~~ u_t(0,x)=u_1(x),  \quad x \in \Omega, \label{1a}\\
&u(t,x)=0, \quad  t\geq 0,  \ \ x\in\partial\Omega,\label{1b}
\end{align}
where  $\Omega$ is a bounded  open subset of $\textbf{R}^n$ $(n\geq 1)$ with smooth boundary $\partial\Omega$ and
\be\label{2}
u_0(x)\in \mathrm{H}_0^1(\Omega), \qquad u_1(x)\in \mathrm{L}^2(\Omega).
\ee
 
The nonlinearity $f(x,u)$ has one of the following forms
\begin{align*}
&(F1)&&f(x,u)=\sum_{i=1}^r a_i(x) |u|^{p_i-1}u + \sum_{j=1}^s b_j(x)|u|^{q_j-1}u,\\
&(F2)&&f(x,u)=a_1(x)|u|^{p_1} + \sum_{i=2}^r a_i(x) |u|^{p_i-1}u + \sum_{j=1}^s b_j(x)|u|^{q_j-1}u.
\end{align*}
Further on, we suppose that the exponents $p_i$, $i=1,...,r$, $q_j$, $j=1,...,s$ and the functions $a_i(x)$, $i=1,...,r$,  $b_j(x)$, $j=1,...,s$ satisfy the following conditions
\be\label{pq}
1<q_1<q_2<
\cdots <q_s<p_1<p_2<\cdots <p_r,
\ee
\be\label{pq-new}
p_r<\infty  \ \text{if} \ n=1,2; \quad
p_r< \frac{n+2}{n-2} \ \text{if} \ n\ge 3,
\ee
\begin{align}
&a_i(x)\in\mathrm{C}(\overbar{\Omega}), \ i=1,\ldots,r, \ \ \ \  b_j(x)\in\mathrm{C}(\overbar{\Omega}), \ j=1,...,s, \label{count}\\
&|a_i(x)|\leq A_i, \ i=1,...,r,\ \ \ \  |b_j(x)|\leq B_j, \ j=1,...,s, \ \ \forall x\in\overbar{\Omega}, \label{5} \\
&a_i(x)\geq 0, \ i=2,\ldots,r, \ \ \ \ b_j(x)\leq 0,  \ j=1,\ldots,s, ~~~\forall~ x\in\overbar{\Omega},\label{signs}\\
&a_1(x)  \ \ \text{is a sign-changing function in} \ \Omega \label{5n}.
\end{align}
The global behavior of the solutions to the wave equation   
\be\label{1k}
\begin{split}
&u_{tt} -  \Delta u = f(u),  \quad t>0, ~~x \in\Omega \subset \textbf{R}^n, ~~n\geq 1 \\
\end{split}
\ee
is subject to comprehensive studies in the last decades.
For negative or zero initial energy, 
the first results for finite time blow up of the solutions to \eqref{1k}, \eqref{1a}--\eqref{2} are obtained in \cite{Ball1,Ball2}, see also
 the earlier  papers \cite{Tsutsumi,Glassey}.
In a series of papers \cite{Levine-1,Levine-2,Levine-3} the author proves blow up of the solutions to abstract nonlinear equations by means of the so-called  concavity method of Levine.
The main idea of the concavity method of Levine  is the reduction of the blow up of the solutions of the wave equation to the blow up of the solutions to the following ordinary differential inequality
\be\label{2k}
\psi''(t)\psi(t) - \gamma \psi'^2(t) \geq 0,~~~~\gamma>1,~~~~ t\geq 0,
\ee
where 
$\psi(t)=\|u(t,\cdot)\|_{\mathrm{L}^2(\Omega)}$. 
Throughout the paper
we  use 
the following short notations for the functions  depending on $t$ and $x$
\begin{align*}
\|u(t)\|_p=\|u(t,\cdot)\|_{\mathrm{L}^p(\Omega)}, \ 0< p\leq \infty,  \quad 
(u(t),v(t))=\int_{\Omega} u(t,x)v(t,x) \, d x
\end{align*}
and for convenience we will write $\|\cdot\|$ instead of $\|\cdot\|_2$.
The proof of the blow up of the solution $\psi(t)$ to \eqref{2k} is based on the concavity of the auxiliary function $z(t)=\psi^{1-\gamma}(t)$.
When $\psi(0)>0$ and $\psi'(0)>0$, then  $z(t)\to 0$ for a finite time. Hence the solution $\psi(t)$ of \eqref{2k}, as well as the solution $u(t,x)$ of \eqref{1k}, blows up for a finite time.
Later on, the concavity method of Levine  is improved in \cite{Straughan}, where  
instead of \eqref{2k} the following more general ordinary differential inequality is considered
\be\label{4k}
\psi''(t)\psi(t) - \gamma \psi'^2(t) \geq -\beta \psi(t) ,~~~~\gamma>1,~\beta>0,~~~ t\geq 0.
\ee
Further generalizations of the concavity method are given in \cite{K-L,Korpusov-1,EJDE}.  In  cited papers the following new ordinary differential inequalities and equation are proposed:
\begin{align}
&\text{\cite{K-L}}&&\psi''(t)\psi(t) - \gamma \psi'^2(t) \geq -2\delta \psi(t)\psi'(t)-\mu \psi^2(t);~~~\label{5k} \\
&\text{\cite{Korpusov-1}}&&\psi''(t)\psi(t) - \gamma \psi'^2(t) \geq -2\delta \psi(t)\psi'(t)-\beta \psi(t);~~~ \label{6k} \\
&\text{\cite{EJDE}}&&\psi''(t)\psi(t) - \gamma \psi'^2(t) = \alpha \psi^2(t)-\beta \psi(t)+H(t),~~~ \label{7k}\\
&&&\gamma>1, ~~\delta\geq 0,~~ \mu\geq 0,~~\beta > 0,~~ \alpha>0,~~H(t)\geq 0,~~ t\geq 0\nonumber.
\end{align}
The finite time blow up of the solution $\psi(t)$ is proved under suitable conditions on the initial data $\psi(0)$ and $\psi'(0)$. 
Inequalities \eqref{2k} --\eqref{6k} and equation \eqref{7k} are applicable to a variety of nonlinear evolution equations
and give different  sufficient conditions on the initial data which guarantee the finite time blow up of the corresponding solutions, see e.g. \cite{Korpusov,Kalantarov,EJDE-2022,Avila-inequality}.

One of the powerful methods for the investigation of  the global behavior of the solutions to nonlinear wave equations is the  so-called  potential well method, introduced
in the pioneering  papers \cite{S,P-S}.
By means of this method,  finite time blow up or global existence of the solutions is completely studied for sub-critical initial energy, $0<E(0)<d$. 
The energy of the initial data $E(0)$ is defined in \eqref{8}, while 
the depth of the potential well $d$ is   given  in \eqref{13}.
In \cite{P-S,Liu} the authors prove that when the Nehari functional is positive $I(u_0)>0$ (see \eqref{I} for definition of $I(u)$), then problem \eqref{1k}, \eqref{1a}--\eqref{2} with 
\be \label{fff}
 f(u)=a |u|^{p-1}u, \quad a>0, ~~p>1
\ee
 has a global solution. If  $I(u_0)<0$, then  the local weak solution  blows up for a finite time. 
In the critical energy case, $E(0)=d$, the global existence holds when $I(u_0)\geq 0$. For $I(u_0)<0$ and $(u_0,u_1)\geq 0$ the solution blows up for a finite time, while for $I(u_0)<0$ and $(u_0,u_1)<0$ then either the solution blows up for a finite time or asymptotically tends to the ground state solution, see 
\cite{Avila1, Avila2,Xu-Quaterly}.
Later on, the results for problem \eqref{1k}, \eqref{1a}--\eqref{2} with  nonlinearity \eqref{fff} and sub-critical initial and critical energy are extended to more general combined power-type nonlinearities with constant coefficients in \cite{Xu-several,Xu-EJDE, Dai}, 
logarithmic nonlinearities \cite{Xu-log}, nonlinearities with variable exponents \cite{Antontsev,Piskin}, etc.
Another application of the potential well method is the study of the wave equations 
\eqref{1k} with   damping terms, see e.g. \cite{Korpusov,Kalantarova,Georgiev,Gazzola,Avila3}, hyperbolic problems with
Neumann's boundary conditions, Neumann-Dirichlet's boundary conditions, Robin's boundary conditions,  dynamical boundary conditions, see e.g. \cite{Kalantarov,Kalantarova,Vitillaro-1, Vitillaro-2} and the references therein.

For super-critical initial energy,  $E(0)>d$, there are only few sufficient conditions for finite time blow up.
In \cite{Straughan}, using \eqref{4k}, 
the finite time blow up of the solutions to abstract wave equations with nonlinearity \eqref{fff} is proved for positive energy under the conditions
\be\label{Straughan-e}
E(0)<\frac{1}{2} \frac{(u_0, u_1)^2}{\| u_0\|^2},  \quad (u_0, u_1)>0,  
\ee
see also \cite{Kalantarov}.
A different technique for proving finite time blow up is suggested by Gazzola and Squissina in \cite{Gazzola}.
For nonlinearity  \eqref{fff} the authors obtain the following sufficient conditions for finite time blow up
\be\label{Gazzola-e}
E(0)<\frac{\mathcal{C}(p-1)}{2(p+1)} \| u_0\|^2,    \quad (u_0, u_1)\geq 0, \quad I(u_0)<0.
\ee
Here $\mathcal{C}$ is the constant defined in \eqref{poin}.
Later on, blow up results under condition 
\eqref{Gazzola-e} are  extended to combined power-type nonlinearities  with constant coefficients, see \cite{Xu-EJDE,Xu-ANZIAM}.
In \cite{Xu-CPAA}, by employing
the adapted concavity method, the author  obtains  a finite time blow up result
of the solution for a class of nonlinear
wave equations with both strongly and weakly damped terms at super-critical
initial energy level. According to this result, the solution of problem \eqref{1k}, \eqref{1a}--\eqref{2}, \eqref{fff} blows up for a finite time if the initial data satisfy the  following conditions 
\be\label{Xu-e}
E(0)<\frac{\mathcal{C}(p_1-1)}{(1+\mathcal{C})(p_1+1)} (u_0, u_1),  \quad (u_0, u_1)>0, \quad I(u_0)<0. 
\ee
Let us also note the results in \cite{Avila-wave} for abstract wave equations with nonlinearities independent of the space and time variables and super-critical initial energy. In \cite{Avila-wave} the author proves finite time blow up of the solutions 
under general abstract sufficient conditions on the initial data. 

The aim of the present paper is the 
 investigation of equation \eqref{1} with more general nonlinear term $f(x,u)$ given by $(F1)$ and $(F2)$.
The coefficients of these nonlinearities depend on the space variables and one of them is a
 sign-changing function. 
We extend  the potential well method from power-type nonlinearities with constant coefficients to nonlinearities 
 $(F1)$ and $(F2)$
with variable coefficients. For this purpose the structure of the Nehari manifold and the properties of the Nehari functional  are studied in details. 
By means of the potential well method the problem for the global behavior of the
weak solutions with non-positive and sub-critical initial energy is completely solved. 
The threshold  between the nonexistence of global in time weak solutions and non-blowing up solutions, is found. 

For arbitrary super-critical initial energy, new sufficient conditions for nonexistence of global in time solutions are found. 
Let us note that one of them implies no restrictions on the sign of
the scalar 
product of the initial data. 
 We compare the sufficient conditions obtained in the paper with the previous ones and demonstrate that they are more general than \eqref{Straughan-e}, \eqref{Gazzola-e}, \eqref{Xu-e}. More precisely,  we construct initial data with arbitrary positive initial energy satisfying the new sufficient conditions, and for which  conditions \eqref{Straughan-e}, \eqref{Gazzola-e}, \eqref{Xu-e} are not fulfilled. Thus, we obtain a wider class of initial data for which the problem has no global in time solutions. 
Finally, uniqueness and existence of local weak solutions to problem \eqref{1}--\eqref{2} is proved for a nonlinearity which is more general than $(F1)$ and $(F2)$.

Recently,  special cases of  problem  \eqref{1}--\eqref{2} is studied in \cite{NTADES2022}. For nonlinearities with  two power-type terms and variable coefficients, only  nonexistence of global in time solutions is proved.

The paper is structured as follows. Some definitions and  preliminary results are recalled in Section~2. Section~3  provides a  precise  analysis of the structure of the Nehari manifold. In Section~4, global existence and nonexistence of global in time solutions  with non-positive or positive sub-critical energy are proved by means of the potential well method. 
Section~5 deals with the case of super-critical initial energy. 
 New sufficient conditions for nonexistence of global weak solutions are obtained and compared with the previous ones. 
The local existence and uniqueness result is formulated and proved in Appendix A.
\section{Preliminary}
Below, we recall some definitions.
\begin{defn}
The function $u(t,x)$ is a weak solution of \eqref{1}--\eqref{2} with nonlinearity $(F1)$ or $(F2)$ in $[0,T_m)\times \Omega$ if 
$$
u(t,x)\in  \mathrm{C}((0,T_m); \mathrm{H}_0^1(\Omega)) \cap \mathrm{C}^1((0,T_m); \mathrm{L}^2(\Omega))
\cap \mathrm{C}^2((0,T_m); \mathrm{H}^{-1}(\Omega)) 
$$
and the identity
$$
\int_{\Omega} u_t(t,x)\eta(x) \,d x +\int_0^t \int_{\Omega} \nabla u(\tau,x) \nabla\eta(x)\, d x \, d \tau =
\int_0^t \int_{\Omega} f(x,u) \eta(x)\, dx \, d \tau +\int_{\Omega} u_1(x)\eta(x) \,d x
$$
holds for every $\eta(x)\in \mathrm{H}_0^1(\Omega)$ and every $t\in [0,T_m)$. 
\end{defn}

\begin{defn}
Suppose  $u(t,x)$ is a  weak solution of \eqref{1}--\eqref{2} with nonlinearity $(F1)$ or $(F2)$ in the maximal existence time interval $[0,T_m)$, $0<T_m\leq \infty$. Then solution  $u(t,x)$
blows up at $T_m$ if
$$
\limsup_{t\to T_m, t<T_m} \|u(t)\| =\infty.
$$
\end{defn}
 
The proof of finite time blow up of the solution to \eqref{1}--\eqref{2} with nonlinearity $(F1)$ and $(F2)$ is based on  the  concavity method of Levine \cite{Levine-1} and its modifications, proposed in \cite{EJDE,ERA}.
The main idea of this method is to replace the
investigation of the solution $u(t,x)$ with the study of the non-negative smooth function $\psi(t)=\|u(t)\|^2$.
We recall some preliminary results, which give sufficient conditions for finite time blow up of the solution of some ordinary differential equations, see  Theorem~2.3 in \cite{EJDE} and Theorem~3.2 in \cite{ERA}. 
\begin{defn}
We say that a non-negative function $\psi(t) \in \mathrm{C}^2([0,T_m))$, $0<T_m\leq\infty$ blows up at $T_m$ if
$$\limsup_{t \rightarrow T_m, t<T_m} \psi(t)= \infty.
$$
\end{defn}

\begin{lem}\label{lm1}(Theorem~3.2 in \cite{ERA})
Suppose $\psi(t) \in \mathrm{C}^2([0,T_m))$ is a nonnegative solution to  the problem
\be\label{11}
\begin{aligned}
&\psi''(t) \psi(t) - \gamma \psi'^2(t) =Q(t), \quad t\in [0,T_m), \quad 0 < T_m \leq \infty,\\
&\gamma >1, \ Q(t) \in \mathrm{C}([0,T_m)),  \ Q(t) \geq 0 \ \text{for}  \ t\in [0,T_m).
\end{aligned}
\ee
If  $\psi(t)$ blows up at $T_m$,  then $T_m < \infty$.
\end{lem}

\begin{lem}\label{lm2}(Theorem~2.3 in \cite{EJDE})
Suppose  $\psi(t)\in  \mathrm{C}^2([0,T_m))$ is a nonnegative solution of the problem
\be\label{12}
\begin{aligned}
&\psi''(t) \psi(t) - \gamma \psi'^2(t) = \alpha\psi^2(t)-\beta\psi(t)+H(t), ~~~~~t\in [0,T_m)~~0 < T_m \leq \infty, \\
&\gamma>1, ~\alpha>0, ~\beta>0, ~~H(t)\in  \mathrm{C}([0,T_m)), ~~H(t)\geq 0~~\text{for}~~t\in [0, T_m).
\end{aligned}
\ee
If  $\psi(t)$ blows up at $T_m$,  then $T_m < \infty$.
\end{lem}

Let us introduce some definitions which are crucial for the extension of the potential well method from 
nonlinearity with constant coefficients to one with variable coefficients.
For every $z\in\mathrm{H}^1_0(\Omega)$ and nonlinearity $f(x,z)$, given by $(F1)$ or $(F2)$, we define
the
potential energy functional $J(z)$ 
$$
J(z):=\frac{1}{2}\|\nabla z\|^2  - \int_{\Omega} \int_{0}^{z(x)} f(x,w) \, dw \, d x,
$$
the Nehari functional $I(z)$
\be\label{I}
I(z):=\|\nabla z\|^2 -  \int_{\Omega} z f(x,z) \, d x,
\ee
the Nehari manifold  $\mathcal{N}$ 
$$
\mathcal{N}=\{ z\in
 \mathrm{H}^1_0(\Omega): \  I(z)=0,  \   ~~~\|\nabla z\| \neq 0\},
$$
and the critical energy constant $d$ (the depth of the potential well) 
\be\label{13}
d=\inf_{z\in \mathcal{N}} J(z).
\ee

From the definitions of $J(z)$ and $I(z)$, and the assumptions \eqref{pq} and \eqref{signs}, we get the following important relation:
\be\label{IJ}
J(z)=\frac{1}{p_1+1}I(z)+\frac{p_1-1}{2(p_1+1)}\|\nabla z\|^2 + B(z),  
\ee
where
$$
\begin{aligned}
B(z)= &\sum_{i=2}^r \frac{p_i-p_1}{(p_1+1)(p_i+1)} \int_{\Omega} a_i(x) |z(x)|^{p_i+1} \, d x  \\ &+
\sum_{j=1}^s \frac{q_j-p_1}{(p_1+1)(q_j+1)} \int_{\Omega} b_j(x) |z(x)|^{q_j+1} \, d x \geq 0.
\end{aligned}
$$
When the functionals $I$, $J$ and $B$ are evaluated on functions depending on $t$ and $x$, then we use the short notations
$I(u(t))=I(u(t,\cdot))$, $J(u(t))=J(u(t,\cdot))$, $B(u(t))=B(u(t,\cdot)).$

In what follows, 
 we will utilize the following embedding theorem for $z \in \mathrm{H}_0^1(\Omega)$:
\be\label{Sob}
\|z\|_q \leq C_q \| \nabla z\|, ~~~ \text{where}~~ 
\left\{\begin{array}{lll}
&1\leq q \leq \infty & \text{if } n=1;\\
&1\leq q <\infty & \text{if } n=2;\\
&1\leq q\leq \frac{2n}{n-2}& \text{if } n\geq 3.
\end{array}
\right.
\ee

\section{Structure of the Nehari manifold}
The Nehari manifold $\mathcal{N}$ and the Nehari functional $I$ play a key role in the application of the potential well method.
In this section we  describe in detail the structure of the Nehari manifold and the relations between functionals $I$, $J$ and the critical energy constant $d$.

\begin{lem}\label{lm3}
Suppose  nonlinearity $f(x,z)$ satisfies $(F1)$,  $z\in\mathrm{H}_0^1(\Omega)$ and $\|\nabla z\|\neq 0$. If
either
\be\label{14}
\sum_{i=2}^{r}\int_{\Omega} a_i(x)\ |z(x)|^{p_i+1} \, dx>0~~~~~~~~~~~~~~~
\ee
or
\be\label{15}
\sum_{i=2}^{r}\int_{\Omega} a_i(x)\ |z(x)|^{p_i+1} \, dx=0 ~~\text{and}~~\int_{\Omega} a_1(x)\ |z(x)|^{p_1+1} \, dx > 0
\ee
holds, then  there exist exactly two constants $\lambda^\ast>0$ and $-\lambda^\ast<0$ such that $\pm\lambda^\ast z \in \mathcal{N}$. 
\end{lem}

\begin{proof}
From \eqref{I} we get
\begin{align*}
I(\lambda z)=&\lambda^2\left(\| \nabla z\|^2 -|\lambda|^{p_1-1}\int_{\Omega} a_1(x) |z(x)|^{p_1+1} \, d x- \sum_{i=2}^r |\lambda|^{p_i-1}\int_{\Omega} a_i(x) |z(x)|^{p_i+1} \, d x \right.\\  -
&\left. \sum_{j=1}^s |\lambda|^{q_j-1}\int_{\Omega} b_j(x) |z(x)|^{q_j+1}\right) \, d x.
\end{align*}
For $\lambda\ne 0$, the equality $I(\lambda z)=0$ is equivalent  to 
\begin{align*}
\| \nabla z\|^2=&|\lambda|^{p_1-1}\int_{\Omega} a_1(x) |z(x)|^{p_1+1} \, d x+\sum_{i=2}^r |\lambda|^{p_i-1}\int_{\Omega} a_i(x) |z(x)|^{p_i+1} \, d x \\  
&+\sum_{j=1}^s |\lambda|^{q_j-1}\int_{\Omega} b_j(x) |z(x)|^{q_j+1} \, dx:=h(\lambda). 
\end{align*}
Since $h(\lambda)$ is an even function in $\lambda$,
we study the behavior of $h(\lambda)$ for $\lambda > 0$.

First, we consider the case:\\
(A)~~~~~~$\ds \sum_{j=1}^s\int_{\Omega} b_j(x) |z(x)|^{q_j+1} \, dx < 0.$\\
Let us represent the function $h(\lambda)$ as
$
h(\lambda)=\lambda^{p_1-1}{h_1}(\lambda),
$ 
where
$$h_1(\lambda)=\int_{\Omega} a_1(x) |z(x)|^{p_1+1} \, d x+ \sum_{i=2}^r \lambda^{p_i-p_1}\int_{\Omega} a_i(x) |z(x)|^{p_i+1} \, d x   
+\sum_{j=1}^s \lambda^{q_j-p_1}\int_{\Omega} b_j(x) |z(x)|^{q_j+1}\, d x.
$$
Under the assumptions of Lemma~\ref{lm3}, we get $h_1'(\lambda)>0,$ $\forall \lambda>0$. Therefore, the function  $h_1(\lambda)$ is  strictly increasing  for $\lambda\in(0,\infty)$.
Moreover,  
$$\lim_{\lambda\to 0^+}h_1(\lambda)=-\infty, 
$$
$$
\lim_{\lambda\to +\infty}h_1(\lambda)=\begin{cases}
+\infty & \text{if }  \ds\sum_{i=2}^r\int_{\Omega}a_i(x) |z(x)|^{p_i+1} \, dx >0; \\[4mm]
 \ds\int_{\Omega}a_1(x) |z(x)|^{p_1+1} \, dx >0 & \text{if } \ds \sum_{i=2}^r\int_{\Omega}a_i(x) |z(x)|^{p_i+1} \, dx =0.
 \end{cases}
$$
Hence, there exists a unique constant $\lambda_0>0$ such that $h_1(\lambda_0)=0$, and consequently, $h(\lambda_0)=0$, $h(\lambda)<0$ for $\lambda\in (0,\lambda_0)$ and $h(\lambda)>0$ for $\lambda\in (\lambda_0,\infty)$. Since $h(\lambda)$ is strictly increasing in $(\lambda_0,\infty)$ and $\lim_{\lambda\to +\infty}h(\lambda)=+\infty$, we get that there exists a unique number  $\lambda^\ast>\lambda_0$ for which $h(\lambda^\ast)=\| \nabla z\|^2$, i.e. $I(\lambda^\ast z)=0$ and  $\lambda^\ast z \in \mathcal{N}$.

Now, we consider the case\\
(B)~~~~~~$\ds\sum_{j=1}^s\int_{\Omega} b_j(x) |z(x)|^{q_j+1} \, dx = 0.$ \\
If condition \eqref{14} with $\int_{\Omega} a_1(x)\ |z(x)|^{p_1+1} \, dx \geq 0$ or condition \eqref{15} holds, 
then function $h(\lambda)$ is a increasing function for $\lambda\in(0,+\infty)$. From $\lim_{\lambda\to 0^+}h(0)=0$ and $\lim_{\lambda\to+\infty}h(\lambda)=+\infty$, we conclude that there exists a unique constant  $\lambda^\ast>0$, such that $h(\lambda^\ast)=\| \nabla z\|^2$, i.e. $\lambda^\ast z \in \mathcal{N}$.

Finally, when condition \eqref{14} with  $\int_{\Omega} a_1(x)\ |z(x)|^{p_1+1} \, dx <0 $ is satisfied,
we investigate the behavior of $h(\lambda)$ analogously to case (A).
Obviously, 
$h_1(\lambda)$ is a increasing function for $\lambda\in(0,\infty)$ and 
$$\lim_{\lambda\to 0^+}h_1(\lambda)=\int_{\Omega}a_1(x) |z(x)|^{p_1+1} \, dx <0,  \ \ \
\lim_{\lambda\to +\infty}h_1(\lambda)=+\infty.
$$
Applying  the same arguments  as in the proof of case (A), we obtain that there exists a unique $\lambda^\ast>0$ such that 
$\lambda^\ast z \in \mathcal{N}$. 
Lemma~\ref{lm3} is proved.

\end{proof}

\begin{lem}\label{lm3n}
Suppose nonlinearity $f(x,z)$ satisfies $(F2)$, $z\in\mathrm{H}_0^1(\Omega)$ and $\|\nabla z\|\neq 0$. 

(i) If
\be\label{14n}
\sum_{i=2}^{r}\int_{\Omega} a_i(x)\ |z(x)|^{p_i+1} \, dx>0,~~~~~~~~~~~~~~~
\ee
then there exist exactly two constants $\lambda^\ast_1>0$ and $\lambda^\ast_2<0$ such that $\lambda^\ast_i z \in \mathcal{N}$, $i=1,2$.

(ii) If
\be\label{15n}
\sum_{i=2}^{r}\int_{\Omega} a_i(x)\ |z(x)|^{p_i+1} \, dx=0 ~~\text{and}~~\int_{\Omega} a_1(x)\ |z(x)|^{p_1} z(x) \, dx \ne 0,
\ee
then there exists a unique constant $\lambda^\ast$ such that $\lambda^\ast z \in \mathcal{N}$.
Moreover, $\lambda^\ast$ is positive if $\int_{\Omega} a_1(x)\ |z(x)|^{p_1}z(x) \, dx > 0$ and  $\lambda^\ast$ is negative
if $\int_{\Omega} a_1(x)\ |z(x)|^{p_1} z(x) \, dx < 0$. 
\end{lem}

The proof of Lemma~\ref{lm3n} is based on the  same idea  as the proof of Lemma~\ref{lm3} and we omit it.

\begin{cor}\label{cor1}
Suppose $f(x,z)$ satisfies either $(F1)$ or $(F2)$. Then the Nehari manifold $\mathcal{N}$ is not empty.
\end{cor}

\begin{proof}

Let $f(x,z)$ be defined by $(F1)$.
From \eqref{count} and \eqref{5n} it follows that there exists $\Omega_1\subset \Omega$ such that $a_1(x)>0$ for $x\in\Omega_1$.
 If we choose  a function $z(x)\in\mathrm{H}_0^1(\Omega_1)$, $z(x)\ne 0$ in $\Omega_1$ and $z(x)\equiv 0$ in $\Omega \setminus \Omega_1$, then either condition \eqref{14} or \eqref{15} is satisfied. Therefore,
Lemma~\ref{lm3} gives us that $\lambda^\ast z \in \mathcal{N}$ for some constant $\lambda^\ast >0$.

The proof of the statement for $f(x,z)$ defined by $(F2)$ is similar to the above one and we omit it. 
\end{proof}
The following Lemma~\ref{lm4} clarifies the relation between  the Nehari functional $I(z)$ and $\| \nabla z \|$. To formulate the result we consider the equation
\be\label{xi0}
\varphi(\xi)=1,  \qquad \text{where} \qquad \varphi(\xi)=\sum_{i=1}^{r} A_i (C_{p_i+1})^{p_i+1} \xi^{p_i-1}.
\ee
The constants $A_i$, $i=1,...,r$ and  $C_{q}$ are defined in \eqref{5} and \eqref{Sob}, respectively. Since $\varphi(\xi)$ is a strictly increasing function for $\xi>0$, then equation \eqref{xi0} has a unique positive root $\xi_0$.

\begin{lem}\label{lm4}
Suppose $f(x,z)$ satisfies either $(F1)$ or $(F2)$, $z\in\mathrm{H}_0^1(\Omega)$ and
$\xi_0$ is the unique positive root of the equation \eqref{xi0}.
Then the following statements are true:
\begin{itemize}
\item[(i)] If $0<\|\nabla z\|<\xi_0$,  then   $I(z)>0$;
\item[(ii)]If $I(z)<0$, then $\|\nabla z\|>\xi_0$;
\item[(iii)]If $I(z)=0$, then either $\| \nabla z \| = 0$ or $\|\nabla z\|\geq \xi_0$.
\end{itemize}
\end{lem}
\begin{proof}
(i) Since $0<\|\nabla z\|<\xi_0$,  then  $\varphi(\|\nabla z\|)<\varphi(\xi_0)=1$. The statement $I(z)>0$ follows from the following chain of inequalities 
\begin{align*}
 \int_{\Omega} z f(x,z) \, d x &\leq \sum_{i=1}^{r} \int_{\Omega} |a_i(x)| |z(x)|^{p_i+1} \, d x +
 \sum_{j=1}^s\int_{\Omega} b_j(x) |z(x)|^{q_j+1} \, d x \\
&\leq \sum_{i=1}^{r}\int_{\Omega} |a_i(x)| \, |z(x)|^{p_i+1} \, d x
\leq \sum_{i=1}^{r} A_i \|z\|_{p_i+1}^{p_i+1}\\
&\leq \sum_{i=1}^{r} A_i C_{p_i+1}^{p_i+1}\|\nabla z\|^{p_i+1}=\|\nabla z\|^2 \varphi(\|\nabla z\|)<\|\nabla z\|^2. 
\end{align*}

(ii)
Condition $I(z)<0$ gives
\begin{align*}
 \|\nabla z\|^2&<\int_{\Omega} z f(x,z) \, d x \leq \sum_{i=1}^{r} \int_{\Omega} |a_i(x)| |z(x)|^{p_i+1} \, d x +
 \sum_{j=1}^s\int_{\Omega} b_j(x) |z(x)|^{q_j+1} \, d x \\
&\leq \sum_{i=1}^{r}\int_{\Omega} |a_i(x)| |z(x)|^{p_i+1} \, d x
\leq \sum_{i=1}^{r} A_i \|z\|_{p_i+1}^{p_i+1}\\
&\leq \sum_{i=1}^{r} A_i C_{p_i+1}^{p_i+1}\|\nabla z\|^{p_i+1}=\|\nabla z\|^2 \varphi(\|\nabla z\|).
\end{align*}
Hence,
$$
\varphi(\|\nabla z\|)> 1=\varphi(\xi_0).
$$
Since $\varphi(\xi)$ is strictly increasing for $\xi>0$, we get 
$$
\|\nabla z\| > \xi_0>0.
$$

(iii) Analogously to (ii) we have $\|\nabla z\|^2\leq \|\nabla z\|^2 \varphi(\|\nabla z\|)$. Hence, either $\| \nabla z \| = 0$ or  $\|\nabla z\|\geq \xi_0$.
Lemma~\ref{lm4} is proved.
\end{proof}

\begin{lem}\label{lm5}
Suppose $f(x,z)$ satisfies either  $(F1)$ or $(F2)$ and  $z\in\mathrm{H}_0^1(\Omega)$. Then the following estimate holds
$$
d\geq\frac{p_1-1}{2(p_1+1)} \xi_0^2>0,
$$
where $\xi_0$ is the unique positive root of \eqref{xi0}.
\end{lem}

\begin{proof}
Corollary~\ref{cor1} gives us  that  $\mathcal{N}\neq \emptyset$, and therefore, the critical constant $d$ is well defined by \eqref{13}.
If $z\in \mathcal{N}$, i.e. $I(z)=0$ and $\|\nabla z\| \ne 0$, then from Lemma~\ref{lm4}~(iii)
 we get
\be\label{r0}
\|\nabla z\|\geq \xi_0>0.
\ee
Thus, from \eqref{13}, \eqref{IJ},   and  \eqref{r0} 
we conclude
$$
d=\inf_{z\in \mathcal{N}} J(z)\geq \frac{p_1-1}{2(p_1+1)} \inf_{z\in \mathcal{N}} \| \nabla z\|^2 \geq \frac{p_1-1}{2(p_1+1)} \xi_0^2>0,
$$
which proves Lemma~\ref{lm5}.
\end{proof}

\begin{lem}\label{lm6}
Suppose $f(x,z)$ satisfies either  $(F1)$ or $(F2)$, $z\in\mathrm{H}_0^1(\Omega)$ and $\| \nabla z\| \neq 0$. If  $I(z)<0$, then there exists a  unique constant $\lambda^\ast > 0$, 
  such that  $I(\lambda^\ast z)=0$ and $\lambda^\ast<1$.
\end{lem}
\begin{proof}
Let $f(x,z)$ be defined by $(F1)$.
From  $I(z)<0$ we get 
$$
\begin{aligned}
0<\| \nabla z\|^2<&\int_{\Omega} z f(x,z) \, d x = \sum_{i=1}^{r} \int_{\Omega} a_i(x) |z(x)|^{p_i+1} \, d x +
 \sum_{j=1}^s\int_{\Omega} b_j(x) |z(x)|^{q_j+1} \, d x \\
&\leq \int_{\Omega} a_1(x) |z(x)|^{p_1+1} \, d x +\sum_{i=2}^{r} \int_{\Omega} a_i(x) |z(x)|^{p_i+1} \, d x.
\end{aligned}
$$
Hence, either condition \eqref{14} or \eqref{15} is satisfied. From Lemma~\ref{lm3} it follows that there exists a unique $\lambda^\ast>0$ such that  $I(\lambda^\ast z)=0$, $I(\lambda z)>0$ for $\lambda\in(0,\lambda^\ast)$ and  $I(\lambda z)<0$ for $\lambda\in(\lambda^\ast,\infty)$. Since $I(z)<0$, we conclude that $\lambda^\ast<1$.

If $f(x,z)$ is defined by $(F2)$, then
$$
\begin{aligned}
0<\| \nabla z\|^2<&\int_{\Omega} z f(x,z) \, d x < \\
&\leq \int_{\Omega} a_1(x) |z(x)|^{p_1} z(x)\, d x +\sum_{i=2}^{r} \int_{\Omega} a_i(x) |z(x)|^{p_i+1} \, d x.
\end{aligned}
$$
From the above inequalities it follows that either condition \eqref{14n} or \eqref{15n} with 
$\int_{\Omega} a_1(x) |z(x)|^{p_1} z(x) \, d x>0$
is satisfied. Then, applying Lemma~\ref{lm3n}, we prove the statement.
\end{proof}

\begin{lem}\label{lm7}
Suppose $f(x,z)$ satisfies either $(F1)$ or $(F2)$, $z\in\mathrm{H}_0^1(\Omega)$ and $\| \nabla z\| \neq 0$. If   $I(z)<0$, then the inequality
\be\label{18}
I(z)<(p_1+1)(J(z)-d)
\ee
holds.
\end{lem}

\begin{proof}
Let $f(x,z)$ be defined by $(F1)$.
From Lemma~\ref{lm6} it follows that there exists a unique constant $0<\lambda^\ast<1$ such that $I(\lambda^\ast z)=0$.
From \eqref{13} and \eqref{IJ} we get the following chain of inequalities:
\begin{align*}
d\leq J(\lambda^\ast z)&=\frac{1}{p_1+1}I(\lambda^\ast z)+\frac{p_1-1}{2(p_1+1)}(\lambda^\ast)^2\|\nabla z\|^2 \\
 &+\sum_{i=2}^r \frac{p_i-p_1}{(p_1+1)(p_i+1)}|\lambda^\ast|^{p_i-p_1} \int_{\Omega} a_i(x) |z(x)|^{p_i+1} \, d x \\
&+\sum_{j=1}^s \frac{q_j-p_1}{(p_1+1)(q_j+1)}|\lambda^\ast|^{q_j-p_1} \int_{\Omega} b_j(x) |z(x)|^{p_j+1} \, d x\\
& <\frac{p_1-1}{2(p_1+1)}\|\nabla z\|^2  
+\sum_{i=2}^r \frac{p_i-p_1}{(p_1+1)(p_i+1)} \int_{\Omega} a_i(x) |z(x)|^{p_i+1} \, d x \\
&+\sum_{j=1}^s \frac{q_j-p_1}{(p_1+1)(q_j+1)} \int_{\Omega} b_j(x) |z(x)|^{p_j+1} \, d x
=J(z)-\frac{1}{p_1+1} I(z).
\end{align*}
The proof of inequality \eqref{18} for nonlinearity $f(x,z)$ defined by $(F2)$ is similar to the above one and we omit it.
Lemma~\eqref{lm7} is proved.
\end{proof}

\section{Sub-critical initial energy}
In this section we 
investigate the global behavior of the  solutions to problem \eqref{1}--\eqref{2} with nonlinearity $(F1)$ or $(F2)$  in the case of sub-critical initial energy $E(0)<d$. For this purpose, first
 we formulate  the local existence and uniqueness result, which is proved in Appendix A. Then we
prove global existence or finite time  blow up of the local solutions depending on the sign of the Nehari functional.
\begin{thm}(Local existence)\label{th1}
There exists $T>0$ such that problem \eqref{1}--\eqref{2} with nonlinearity $(F1)$ or $(F2)$ admits a unique local weak solution 
 in  $[0,T]$ provided the additional restriction $p_r<\frac{n}{n-2}$ for $n\geq 3$. Moreover, if $T_m= \sup\{T>0: \ u(t,x)~ \text{exists~ on~} [0,T]\times \Omega\}\leq \infty$,  then
\begin{itemize}
\item [(i)] every local solutions  $u(t,x)$ satisfies the conservation law 

\be\label{cons}
E(0)=E(t) \ \ \text{for} \ t\in[0,T_m),
\ee
where the energy functional $E(t)$ is defined by
\be \label{8}
\begin{split}
E(t):=E(u(t),u_t(t) )=&\frac{1}{2}\left(
\|u_t(t)\|^2 +\|\nabla u(t)\|^2 \right)
- \int_{\Omega} \int_0^{u(t,x)} f(x,z) \, dz \, d x;
\end{split}
\ee

\item [(ii)]
$$
\text{if} \qquad \limsup_{t\to T_m, t<T_m} \|\nabla u(t)\| <\infty,
\qquad \text{then} \qquad T_m=\infty.
$$
\end{itemize}
\end{thm}

Theorem~\ref{th1} is a special case  of Theorem~\ref{th1n} in Appendix A, where the result is formulated and proved for problem \eqref{1}--\eqref{2} with a  more general nonlinearity $f(x,u)$ than  
$(F1)$ and $(F2)$.

In the energy sub-critical case, $E(0)<d$, we employ the potential well method in proving global existence or finite time blow up of the weak solutions.
In the framework of this approach there are two important subsets of $\mathrm{H}^1_0(\Omega)$:
$$
W=\left\{z\in \mathrm{H}^1_0(\Omega):I(z)>0 \right\}\cup\left\{ 0\right\}, \ \ \ 
V=\left\{z\in \mathrm{H}_0^1(\Omega):I(z)<0 \right\}.
$$
For $E(0)<d$, in the following theorem we prove the invariance of $W$ and $V$  under the flow of \eqref{1}--\eqref{2}, i.e. the sign preserving properties of the Nehari functional $I(z)$.

\begin{thm}\label{th2}
Suppose $u(t,x)$ is the weak solution of \eqref{1}--\eqref{2}   with nonlinearity $(F1)$ or $(F2)$ in the maximal existence time interval $[0,T_m)$, $0<T_m\leq \infty$. Then:
\begin{itemize}
\item[(i)]~ if $0<E(0)<d$ and $u_0(x) \in W$, then $u(t,x) \in W$ for every $t \in[0,T_m)$;

\item[(ii)]~ if $0<E(0)<d$ and $u_0(x) \in V$, then $u(t,x) \in V$ for every $t \in[0,T_m)$.
\end{itemize}
\end{thm}
\begin{proof}
(i) Assume, for sake of contradiction, that there exists some $t_1$ such that $u(t_1,x)\notin  W$. Since $I(u(t))$ is continuous with respect to $t$, we get that there exists $t_0\in (0,t_1)$ such that $u(t_0,x) \in \partial W=\left\{z\in \mathrm{H}^1_0(\Omega):I(z)=0 \right\}$. Note, that $0\notin \partial W$. Indeed,  
from Lemma~\ref{lm4}~(i), we have that if $0<\|\nabla z \|<\xi_0$, then
$I(z)>0$, which implies the inclusion $B_{\xi_0}=\{ z\in \mathrm{H}^1_0(\Omega):  0<\| \nabla z\|<\xi_0\} \subset W$.
Therefore, $u(t_0,x) \in \partial W$ reads $I(u(t_0))=0$ with $\|\nabla u(t_0)\| \ne 0$, i.e. 
$u(t_0,x) \in \mathcal{N}$.
From the conservation law \eqref{cons} and the definition of $d$, we obtain the following impossible chain of inequalities
$$
d=\inf_{z\in \mathcal{N}} J(z)\leq J(u(t_0))=E(t_0)- \frac{1}{2} \|u_t(t_0)\|^2\leq E(0)<d.
$$
So, $u(t,x)\in W$ for every $t\in [0,T_m)$ and the statement (i) is proved.

(ii) The proof is similar to the proof of (i). Assume, for sake of contradiction, that  $t_1$ is the first time such that 
$u(t_1,x)\notin V$, i.e.  $u(t,x)\in V$ for every $t\in[0,t_1)$ and $I(u(t_1))=0$.
From Lemma~\ref{lm4}~(ii) it follows that $ \| \nabla u(t)\| >\xi_0$ for every $t\in[0,t_1)$. Hence,  $ \| \nabla u(t_1)\|  \geq \xi_0>0$ and
 $u(t_1,x) \in \mathcal{N}$. Similar to (i), we get $d\leq J(u(t_1))\leq E(0)$, which contradicts the assumption $E(0)<d$. Therefore, we conclude that $u(t,x)\in V$ for every $t\in [0,T_m)$, which proves statement (ii).
\end{proof}

Further on, we will use the following helpful  relations between the energy, the potential functional and the Nehari functional (see \eqref{8} and \eqref{IJ}):
\be\label{EJ}
\begin{aligned}
E(0)=&E(t)=\frac{1}{2}\|u_t(t)\|^2 +J(u(t))
=\frac{1}{p_1+1}I(u(t))+\frac{1}{2} \|u_t(t)\|^2 \\
&+\frac{p_1-1}{2(p_1+1)}\|\nabla u(t)\|^2 +
 B(u(t)), \quad  B(u(t)) \geq 0, \  t\in [0,T_m).
\end{aligned}
\ee
The following lemma  allows us to use the potential well method in the case of non-positive energy. 
\begin{lem}\label{E0}
Suppose $E(0)<0$ or $E(0)=0$ and $\| \nabla u_0\| \ne  0$. Then every weak solution of \eqref{1}--\eqref{2} with nonlinearity $(F1)$ or $(F2)$  belongs to $V$.
\end{lem}
\begin{proof}
If $E(0)<0$, it is obvious from \eqref{EJ}, that $I(u(t))<0$, i.e. $u(t) \in V$ for every $t \in[0, T_m)$. 
Now we consider the case $E(0)=0$ and $\| \nabla u_0\| \ne  0$. Since $I(u_0)<0$, from Lemma~\ref{lm4}~(ii) we obtain that $\|\nabla u_0\| >\xi_0$.  
We will prove that $\|\nabla u(t)\| \geq \xi_0$ for every $t\in[0,T_m)$. Otherwise,  there exists $t_0 \in(0, T_m)$ such that 
$0<\|\nabla u(t_0)\| < \xi_0$ and Lemma~\ref{lm4}~(i) gives us $I(u(t_0))>0$, which implies $E(t_0)=E(0)>0$. This contradicts the assumption $E(0)=0$.
Therefore, $\|\nabla u(t)\| \geq \xi_0$ and $I(u(t))<0$ for every $t \in[0,T_m)$. The proof is completed.
\end{proof}

\begin{thm}\label{th3}
Suppose $0<E(0)<d$   and $u_0(x)\in W$.
 Then problem \eqref{1}--\eqref{2} with nonlinearity $(F1)$ or $(F2)$
has no blowing up weak solutions. Moreover,  if the additional restriction  $p_r<\frac{n}{n-2}$ for $n\geq 3$ holds, then the  problem  
 has a unique  global weak solution.   
\end{thm}
\begin{proof}
Suppose by contradiction that problem \eqref{1}--\eqref{2} with nonlinearity nonlinearity $(F1)$ or $(F2)$
has  blowing up weak solution $u(t,x)$ defined in the maximal existence time interval $[0, T_m)$, $0< T_m\leq \infty$. 
If $u_0(x)\in W$, from Theorem~\ref{th2}~(i) it follows that $u(t) \in W$ for every $t\in(0,T_m)$.
From  \eqref{EJ},   the estimate
\be\label{50}
\begin{split}
d > E(0)=E(t)
\geq \frac{p_1-1}{2(p_1+1)}\|\nabla u(t)\|^2
\end{split}
\ee
holds for every $t\in(0,T_m)$. Applying  the embedding theorem \eqref{Sob}, we get 
$$
\| u(t)\|^2 \leq C_2^2 \frac{2(p_1+1)}{p_1-1}d,
$$
which contradicts our assumption that $u(t,x)$ blows up.

If additionally $p_r<\frac{n}{n-2}$ for $n\geq 3$, then according to Theorem~\ref{th1}, problem \eqref{1}--\eqref{2} with nonlinearity $(F1)$ or $(F2)$   has a unique local weak solution $u(t,x)$.
Since  estimate  \eqref{50} is satisfied in the maximal existence time interval of $u(t,x)$, from  Theorem~\ref{th1}~(ii) we have that the  weak solution is globally defined. Thus,  Theorem~\ref{th3} is proved.
\end{proof}

\begin{thm}\label{th4}
If either  $E(0)< 0$ or $0\leq E(0)<d$ and $u_0(x) \in V$, then problem \eqref{1}--\eqref{2} with nonlinearity $(F1)$ or $(F2)$ has no global in time weak solution 
 $u(t,x)$.

\end{thm}

\begin{proof}

Assume, for sake of contradiction, that $u(t,x)$ is globally defined. It is easy to obtain
for function $\psi(t)=\|u(t) \|^2$ that
\be\label{52}
\psi'(t)=2(u(t),u_t(t)),  \quad \psi{''}(t)=2\|u_t(t)\|^2 -2I(u(t)).
\ee
Moreover, using \eqref{52} and \eqref{EJ}, we come to the following expression for $\psi{''}(t)$:
\be\label{51}
\psi''(t)=(p_1+3)\|u_t(t)\|^2 - 2(p_1+1)E(0) + (p_1-1) \|\nabla u(t)\|^2
 + 2(p_1+1) B(u(t)).
\ee
First, we show that under the assumptions of the theorem, there exists a constant $M>0$, such that
\be\label{M}
\psi''(t)\geq M>0  \quad \text{for} \ \ t\in[0,\infty).
\ee

\noindent
\emph{Case 1.} $E(0)< 0$ \\
From  \eqref{51}  we directly obtain that \eqref{M} holds with $M=- 2(p_1+1)E(0)$.

\noindent
\emph{Case 2.} $E(0)=0$\\
Since $u_0 \in V$, i.e. $I(u_0)<0$, from  Lemma~\ref{lm4}~(ii) it follows that $\| \nabla u_0\|>\xi_0>0$. According to Lemma~\ref{E0} we have $I(u(t))<0$ for every $t \in[0,T_m)$.
Expression \eqref{51} and Lemma~\ref{lm4}~(ii) give us \eqref{M} with $M= (p_1-1)\xi_0^2$.

\noindent
\emph{Case 3.}  $0<E(0)<d$\\
By \eqref{52}, 
Theorem~\ref{th2}(ii) and Lemma~\ref{lm7}, we have the  following estimate   
$$
\psi''(t)\geq - 2I(u(t))
>2(p_1+1)(d-J(u(t))
\geq 2(p_1+1) (d-E(0))>0, 
$$
i.e. $M=2(p_1+1) (d-E(0))$.
Therefore, in all three cases the inequality \eqref{M} holds.
Integrating \eqref{M} twice, we obtain
$$
\psi(t)\geq M t^2 +\psi'(0) t +\psi(0),
$$
which yields to $\lim_{t\to\infty}\psi(t)=\infty$.

On the other hand,  $\psi(t)$ satisfies equation \eqref{11}, 
where $\gamma=\frac{p_1+3}{4}>1$ and, depending on the energy level, $Q(t)$ has one of the following forms:

\noindent
\emph{Case 1. and Case 2.} $E(0)\leq 0$
\begin{align*}
Q(t)=&(p_1+3)\left(\|u(t)\|^2 \|u_t(t)\|^2 - (u(t),u_t(t))^2\right) 
     -2(p_1+1) E(0)\|u(t)\|^2\\
		&+ \left[(p_1-1) \| \nabla u(t)\|^2 + 2(p_1+1) B(u(t))\right] \|u(t)\|^2
		 \geq 0~~\text{for}~ t\in[0,\infty)
\end{align*}

\noindent
\emph{Case 3.} $0<E(0)<d$
\begin{align*}
Q(t)=&(p_1+3)\left(\|u(t)\|^2 \|u_t(t)\|^2 - (u(t),u_t(t))^2\right) \\
     &+ 2\left[(p_1+1)(J(u(t))-
E(0))-I(u(t))\right] \|u(t)\|^2 \geq 0~~\text{for}~ t\in[0,\infty).
\end{align*}
From  Lemma~\ref{lm1}, function $\psi(t)$ blows up for a finite time. Hence  $u(t,x)$ also blows up for a finite time,  which contradicts our assumption that $u(t,x)$ is globally defined.
The proof of Theorem~\ref{th4} is completed. 
\end{proof}

\begin{rem}
Suppose that the additional restriction $p_r< \frac{n}{n-2}$ for $n\geq 3$ holds. Then,  according to Theorem~\ref{th1}, problem \eqref{1}--\eqref{2} with nonlinearity $(F1)$ or $(F2)$ has a unique local solution, which blows up for a finite time (see Theorem~\ref{th4}).
\end{rem}

\section{Super-critical initial energy}
Below, we will use the notation
\be\label{poin}
\mathcal{C}=C_2^{-2},
\ee
where $C_2$ is the constant defined in \eqref{Sob} for $q=2$. The constant $\mathcal{C}$ can be interpreted  as the constant in the Poincar\'{e} inequality
\be\label{poin1}
\| \nabla z\|^2\geq \mathcal{C} \|z\|^2 \ \ \text{for} \ z\in  \mathrm{H}_0^1(\Omega).
\ee

\subsection{Nonexistence of global in time solutions  when   the scalar product of the initial data has arbitrary sign}
\begin{thm}\label{th6}
If 
\be\label{56}
0<E(0)<\frac{\mathcal{C}(p_1-1)}{2(p_1+1)}\|u_0\|^2 + \frac{\sqrt{\mathcal{C} (p_1-1)}}{p_1+1} (u_0,u_1)
\ee
with
constant $\mathcal{C}$   defined in \eqref{poin},
then   problem \eqref{1}--\eqref{2} with nonlinearity $(F1)$ or $(F2)$ has no global in time weak solution 
 $u(t,x)$.
\end{thm}

\begin{proof}
Analogously to the proof of Theorem~\ref{th4}, we assume that $u(t,x)$ is globally defined.
From \eqref{51},  we get that $\psi(t)=\|u(t)\|^2$
 is a solution  of the equation
\be\label{53}
\psi''(t)=\alpha \psi(t) - \beta + G(t),
\ee
where 
\be\label{ab}
\alpha=\mathcal{C}(p_1-1)>0, ~~~\beta=2(p_1+1)E(0)>0,
\ee
 and
$$
G(t)=(p_1-1)(\|\nabla u(t)\|^2 - \mathcal{C} \|u(t)\|^2) + (p_1+3)\|u_t(t)\|^2 +2(p_1+1)B(u(t)).
$$
Equation  \eqref{53} has the following classical solution
\begin{align*}
\psi(t)=&\frac{1}{2}\left(\psi(0)+\frac{1}{\sqrt{\alpha}} \psi'(0)-\frac{\beta}{\alpha}\right) e^{\sqrt{\alpha}t}
+ \frac{1}{2}\left(\psi(0)-\frac{1}{\sqrt{\alpha}} \psi'(0)-\frac{\beta}{\alpha}\right) e^{-\sqrt{\alpha}t} \\
&+ \frac{\beta}{\alpha}+\frac{1}{\sqrt{\alpha}} \int_{0}^t G(s) \sinh(\sqrt{\alpha}(t-s)) \, d s.
\end{align*}
Since \eqref{56} is equivalent to
$$
\psi(0)+\frac{1}{\sqrt{\alpha}} \psi'(0)-\frac{\beta}{\alpha}>0
$$
and $G(t)\geq 0$ for $t\in[0,\infty)$, we  conclude that $\lim_{t\to\infty} \psi(t)=\infty$.

Furthermore, $\psi(t)$ is also a solution to the equation \eqref{12} with $\alpha$ and $\beta$ defined in \eqref{ab}, 
 $\gamma=\frac{p_1+3}{4}$ and 
\begin{align*}
H(u(t))=&(p_1+3)(\|u_t(t)\|^2\|u(t)\|^2 - (u(t),u_t(t))^2) 
+(p_1-1) (\| \nabla u(t)\|^2 - \mathcal{C}\|u(t)\|^2) \|u(t)\|^2 \\
&+ 2(p_1+1) B(u(t)) \|u(t)\|^2 \geq 0 ~~\text{for} \quad t\in[0,\infty).
\end{align*}
Thus, according to  Lemma~\ref{lm2}, the function $\psi(t)$, as well as  $u(t,x)$, blows up for a finite time, which contradicts our assumption $T_m=\infty$.
Theorem~\ref{th6} is proved.
\end{proof}

\subsection{Nonexistence of global in time solutions  when   the scalar product of the initial data has positive sign}
\begin{thm}\label{th7}
If  
\be \label{2m}
\|u_0\|\ne 0,~~~
(u_0,u_1)\geq 0, \quad 
0<E(0)<\frac{\mathcal{C}(p_1-1)}{2(p_1+1)}\|u_0\|^2 + \frac{1}{2}\frac{(u_0,u_1)^2}{\|u_0\|^2}
\ee
with
constant $\mathcal{C}$   defined in \eqref{poin},
then   problem \eqref{1}--\eqref{2} with nonlinearity $(F1)$ or $(F2)$ has no global in time weak solution 
 $u(t,x)$.
\end{thm}

\begin{proof}

\emph{Step 1:}
First, we will prove that the functions
$$\psi(t)=\|u(t)\|^2 ~~~\text{and}~~~ \phi(t)=\frac{(u(t),u_t(t))^2}{\|u(t)\|^2}=\frac{1}{4}\frac{(\psi'(t))^2}{\psi(t)}$$
 are strictly increasing  for $t\in(0,T)$, $T\leq T_m$, provided $I(u(t))<0$ for $t\in[0,T)$, i.e.
 $\psi'(t)>0$ and $\phi'(t)>0$ for $t\in(0,T)$. 

From \eqref{52} and the condition $\psi'(0)=2(u_0(x),u_1(x))\geq 0$ it follows that $\psi''(t)>0$ for $t\in[0,T)$, i.e. $\psi'(t)$ is strictly increasing and $\psi'(t)>0$ for $t\in (0,T)$.

For $\phi'(t)$ we get
\be\label{1m}
\phi'(t)=\frac{d}{dt} \left(\frac{1}{4}\frac{(\psi'(t))^2}{\psi(t)} \right)=\frac{1}{4}\frac{\psi'(t)}{\psi^2(t)} \left[2\psi(t) \psi''(t)-(\psi'(t))^2\right].
\ee
Since
$$
2\psi(t) \psi''(t)-(\psi'(t))^2=4\left[\|u(t)\|^2 \|u_t(t)\|^2 - (u(t),u_t(t))^2\right] -2 \|u(t)\|^2 I(u(t))> 0$$
and $\psi'(t)>0$ for $t\in (0,T)$, formula \eqref{1m} yields  $\phi'(t)>0$ for $t\in(0,T)$. Therefore, we get that $\psi(t)$ and $\phi(t)$ are strictly increasing functions in $(0,T)$ as long as $I(u(t))<0$.

\emph{Step 2:} Now we will prove that if the initial data $u_0(x)$ and $u_1(x)$ satisfy \eqref{2m}, then $I(u(t))<0$ for every $t\in[0,T_m)$.
From \eqref{EJ} and \eqref{IJ} we have the identity
\be\label{3m}
\begin{split}
\frac{I(u(t))}{p_1+1}=&E(t)-\frac{\mathcal{C}(p_1-1)}{2(p_1+1)}\|u(t)\|^2 - \frac{1}{2}
\frac{(u(t),u_t(t))^2}{\|u(t)\|^2}\\
 &+
\frac{(p_1-1)}{2(p_1+1)}(\mathcal{C} \|u(t)\|^2 -\|\nabla u(t)\|^2)\\
&+\frac{1}{2} \left(\frac{(u(t),u_t(t))^2}{\|u(t)\|^2} - \|u_t(t)\|^2\right)
-B(u(t)).
\end{split}
\ee
By means of \eqref{2m} and \eqref{poin1}, it is easy to obtain that $I(u_0)<0$.
Assume, for sake of contradiction, that $t_0$ is the first time such that $I(u(t_0))=0$ and $I(u(t))<0$ for 
$t\in[0,t_0)$. From the continuity and monotonicity of $\psi(t)$ and $\phi(t)$ in $(0,t_0]$, \eqref{2m} and \eqref{3m}, we get the following impossible chain of inequalities:
\begin{align*}
0=\frac{I(u(t_0))}{p_1+1}\leq &E(0)-\frac{\mathcal{C}(p_1-1)}{2(p_1+1)}\|u(t_0)\|^2 - \frac{1}{2}
\frac{(u(t_0),u_t(t_0))^2}{\|u(t_0)\|^2}\\
 & \leq
E(0)-\frac{\mathcal{C}(p_1-1)}{2(p_1+1)}\|u_0\|^2 - \frac{1}{2}
\frac{(u_0,u_1)^2}{\|u_0\|^2}<0.
\end{align*}
Therefore,  $I(u(t))<0$ for every $t\in[0,T_m)$.

\emph{Step 3:}
Analogously to the proofs of Theorem~\ref{th4} and  Theorem~\ref{th6}, we assume that $u(t,x)$ is globally defined. Using \eqref{51},  we have
$$
\psi''(t)=\mathcal{C}(p_1-1)\|u(t)\|^2 +(p_1+1)
\frac{(u(t),u_t(t))^2}{\|u(t)\|^2 }
-2(p_1+1)E(0) + S(u(t)),
$$
where
\begin{align*}
S(u(t))=&(p_1+1)\left(\|u_t(t)\|^2 - \frac{(u(t),u_t(t))^2}{\|u(t)\|^2}\right) \\
&+(p_1-1) (\| \nabla u(t)\|^2 - \mathcal{C}\|u(t)\|^2)  + 2 \|u_t(t)\|^2 + 2(p_1+1) B(u(t)) \geq 0.
\end{align*}
Step~1, Step~2,  and \eqref{2m}  give us
$$
\psi''(t)
\geq \mathcal{C}(p_1-1)\|u_0\|^2 +(p_1+1)
\frac{(u_0,u_1)^2}{\|u_0\|^2 }
-2(p_1+1)E(0)>0.
$$
Similar to the proof of Theorem~\ref{th4}, we obtain that $\lim_{t\to\infty} \psi(t)=\infty$.

\emph{Step 4:}
In the proof of Theorem~\ref{th6} it is shown that $\psi(t)$ satisfies the conditions of Lemma~\ref{lm2}, which means that $\psi(t)$, as well $u(t,x)$, blows up for a finite time. This contradicts our assumption $T_m=\infty$.
Theorem~\ref{th7} is proved.
\end{proof}

\subsection{Existence of initial data satisfying the conditions of Theorem~\ref{th6}  and Theorem~\ref{th7}}

First, we demonstrate that the set of initial data that satisfy  conditions of Theorem~\ref{th6}  or Theorem~\ref{th7} is not empty, i.e.
we construct explicitly initial data  with arbitrary positive energy, for which  assumption \eqref{56} or \eqref{2m} holds. 

For simplicity, we consider the following  single power nonlinearity with a variable sign-changing coefficient
\be\label{sn}
f(x,u)=a_1(x) |u|^{p_1-1} u,  
\ee
where
$$
1<p_1<\infty \ \text{if} \ n=1,2;  \quad 1 < p_1 \leq \frac{n+2}{n-2} \ \text{if} \ n\ge 3, 
$$
$a_1(x)\in\mathrm{C}(\overbar{\Omega})$ and  
$a_1(x)$ is a sign-changing function in $\Omega$.
Let us note that \eqref{sn} is  a special case of nonlinearity $(F1)$.

\begin{prop}\label{pr2}
For every positive constant $K$ there exist infinitely many initial data $u_0^K(x)$, $u_1^K(x)$ that have energy $E(u_0^K, u_1^K)=K$ and satisfy condition \eqref{56}. Moreover, problem \eqref{1}--\eqref{2}, \eqref{sn} with initial data $u_0^K(x)$, $u_1^K(x)$ has no global in time weak solution. 
\end{prop}
\begin{proof}
Let  $w(x)\in \mathrm{H}_0^1(\Omega)$ and $v(x)\in \mathrm{L}^2(\Omega)$ be fixed functions satisfying
\be
\|\nabla w\| \ne 0,  \ \|v\| \ne 0, \  (w,v)=0, \ \int_{\Omega} a_1(x) |w(x)|^{p_1+1} \, d x>0.
\label{4.1}
\ee
Let $K$ be an arbitrary positive constant.
We construct  initial data $u_0^K$, $u_1^K$ in the following way:
\be\label{4.2}
u_0^K(x)=  \mu w(x), \quad u_1^K(x)= \mu \sigma w(x) +  \eta v(x),
\ee
where 
$$
\sigma \in \left(-\frac{\sqrt{\mathcal{C}(p_1-1)}}{2},\infty\right) \setminus \{0\}
$$
and 
the constants $\mu>0$ and $\eta >0$ will be chosen below.

Straightforward  computations give us the following formulas for  energy $E(0)$ and different norms of $u_0^K$ and $u_1^K$ in terms of the norms of  $w$ and $v$:
\begin{align*}
& \|u_0^K\|^2= \mu ^2\|w\|^2,~~
\|u_1^K\|^2= \mu ^2 \sigma^2 \|w\|^2 + \eta^2 \|v\|^2,\\
&(u_0^K,u_1^K)=\mu^2 \sigma \|w\|^2, ~~
\|\nabla u_0^K\|^2= \mu^2 \|\nabla w\|^2, \\
&E(u_0^K, u_1^K)=E(0)= R(\mu) + \frac{\eta^2}{2} \|v\|^2,~~~~\text{where}\\
&R(\mu):= \frac{\mu^2 \sigma^2}{2} \|w\|^2 + \frac{\mu^2}{2} \|\nabla w \|^2
- \frac{\mu^{p_1+1}}{p_1+1} \int_{\Omega} a_1(x) |w(x)|^{p_1+1} \, d x.
\end{align*}
Let us denote by $\mu_0$ the unique positive root of the equation $R(\mu)=0$, i.e.
$$
\mu_0=\left(\frac{(p_1+1) (\sigma^2 \|w\|^2+\|\nabla w\|^2)}{2\int_{\Omega} a_1(x) |w(x)|^{p_1+1} \, d x}\right)^{\frac{1}{p_1-1}}.
$$
Note that for all $\mu>\mu_0$ the inequality $R(\mu) <0$ holds.

From \eqref{4.1} and \eqref{4.2} it follows   that the sign of $(u_0^K, u_1^K)$  coincides with the sign of the constant $\sigma$.
Now, we  set $\mu=\mu_1$ and $\eta=\eta_1$, where
$$
\mu_1 > \max\left\{\mu_0, \frac{\sqrt{K}}{\sqrt{\frac{\sqrt{\mathcal{C}(p_1-1)}}{(p_1+1)} \left(\frac{\sqrt{\mathcal{C}(p_1-1)}}{2} + \sigma\right)}} \|w\|^{-1}\right\}, \ \ 
\eta_1=\sqrt{2(K-R(\mu_1))} \|v\|^{-1}.
$$
The constant $\eta_1$ is well defined since $R(\mu_1)<0$. 

Let us check that the initial data \eqref{4.2} with the above choice of the constants $\mu_1$ and $\eta_1$ satisfy condition \eqref{56} and the initial energy is equal to the constant K. 
Indeed, 
$$
E(0)=R(\mu_1)+ \frac{\eta_1^2}{2} \|v\|^2=K
$$
and
\begin{gather*}
\frac{\mathcal{C}(p_1-1)}{2(p_1+1)}\|u_0\|^2+\frac{\sqrt{\mathcal{C}(p_1-1)}}{(p_1+1)}(u_0,u_1)\\
=\mu_1^2 \frac{\sqrt{\mathcal{C}(p_1-1)}}{(p_1+1)} \left(\frac{\sqrt{\mathcal{C}(p_1-1)}}{2} + \sigma\right)\|w\|^2 > K=E(0).
\end{gather*}
Thus initial data \eqref{4.2}, with already chosen constants $\mu_1$ and  $\eta_1$,  satisfy  all conditions
 of Theorem~\ref{th6}. Moreover, these  initial data have arbitrary  positive energy $E(0)= K$.
In this way if we take $K>d$ we find a wide class of initial data \eqref{4.2} with super-critical energy $E(0)>d$ for which the blow up result of Theorem~\ref{th6} is valid.
\end{proof}

\begin{prop} \label{pr1}
For every positive constant $K$ there exist infinitely many initial data $u_0^K(x)$, $u_1^K(x)$ that have energy $E(u_0^K, u_1^K)=K$ and satisfy condition \eqref{2m}. Moreover, problem \eqref{1}--\eqref{2}, \eqref{sn} with initial data $u_0^K(x)$, $u_1^K(x)$ has no global in time weak solution. 
\end{prop}

\begin{proof}
The proof is similar to the proof of Proposition~\ref{pr2}. 
We choose initial data $u_0^K$, $u_1^K$ by formulas \eqref{4.2} and set $\sigma=1$.

It is easy to see that $(u_0^K, u_1^K)>0$  for every $\mu >0$ and $\eta >0$. 
We fix  constants  $\mu$ and $\eta$ in \eqref{4.2} as follows:
$$
\mu=\mu_2 > \max\left\{\mu_0, \frac{\sqrt{2K}}{\sqrt{\frac{\mathcal{C}(p_1-1)}{p_1+1}+1}} \|w\|^{-1}\right\}, \ \ 
\eta=\eta_2=\sqrt{2(K-R(\mu_1))} \|v\|^{-1}.
$$
For  initial data \eqref{4.2} with  already chosen  constants $\mu_2$ and $\eta_2$  we have 
$$
E(u_0^K, u_1^K)=E(0)= R(\mu_2) + \frac{\eta_2^2}{2} \|v\|^2=K,
$$
and
condition \eqref{2m} holds, i.e. 
$$
\frac{\mathcal{C}(p_1-1)}{2(p_1+1)}\|u_0\|^2+\frac{1}{2}\frac{(u_0,u_1)^2}{\|u_0\|^2}=
\frac{\mu_2^2}{2}\left(\frac{\mathcal{C}(p_1-1)}{p_1+1}+1\right)\|w\|^2 > K=E(0).
$$
Hence initial data \eqref{4.2} with $\sigma=1$, $\mu=\mu_2$ and  $\eta=\eta_2$  satisfy  all conditions
 of Theorem~\ref{th7} and have arbitrary  positive energy $E(0)= K$.
According to Theorem~\ref{th7}, problem \eqref{1}--\eqref{2}, \eqref{sn} with initial data $u_0^K(x)$, $u_1^K(x)$ has no global in time weak solution. 
\end{proof}

\subsection{Comparison of the sufficient conditions for nonexistence of global solutions}

Finally, we compare  sufficient condition  \eqref{2m} in Theorem~\ref{th7} with conditions \eqref{Straughan-e}, \eqref{Gazzola-e} and \eqref{Xu-e}. More precisely, we prove that 
condition  \eqref{2m} is more general than \eqref{Straughan-e}, \eqref{Gazzola-e} and \eqref{Xu-e}. 
\begin{prop}\label{pr3}
If  
initial data $u_0$, $u_1$ satisfy one of the conditions \eqref{Straughan-e}, \eqref{Gazzola-e} and \eqref{Xu-e}, then they satisfy  condition \eqref{2m}. 
Moreover,  for these initial data, problem \eqref{1}--\eqref{2} with nonlinearity $(F1)$ or $(F2)$  has no global in time weak solution.
\end{prop}

\begin{proof}
In the proof of Theorem~\ref{th7} we show that $I(u_0)<0$ when initial data satisfy \eqref{2m}. Taking this fact into account, we  can easily conclude
that condition \eqref{2m} is more general than conditions \eqref{Straughan-e} and \eqref{Gazzola-e}.
Now we demonstrate that 
if the initial data $u_0$,  $u_1$ satisfy \eqref{Xu-e}, then 
\eqref{2m}  necessarily holds. The proof of this statement follows 
from the following chain of inequalities:
\begin{gather*}
\frac{\mathcal{C}(p_1-1)}{2(p_1+1)}\|u_0\|^2 + \frac{1}{2}\frac{(u_0,u_1)^2}{\|u_0\|^2}=\\
\frac{1}{2}\left( \frac{(u_0,u_1)}{\|u_0\|}  - \frac{\mathcal{C}(p_1-1)}{(1+\mathcal{C})(p_1+1)} \|u_0\| \right)^2 +\frac{\mathcal{C}(p_1-1)}{(1+\mathcal{C})(p_1+1)} (u_0,u_1) \\
+\frac{\mathcal{C}(p_1-1)}{2(p_1+1)} \left( 1-\frac{\mathcal{C}(p_1-1)}{(1+\mathcal{C})^2(p_1+1)} \right) \|u_0\|^2 >
\frac{\mathcal{C}(p_1-1)}{(1+\mathcal{C})(p_1+1)} (u_0,u_1).
\end{gather*}
Hence, condition \eqref{2m} holds for all initial data $u_0$, $u_1$, satisfying one of the conditions \eqref{Straughan-e}, \eqref{Gazzola-e} or \eqref{Xu-e}. According to Theorem~\ref{th7} 
problem \eqref{1}--\eqref{2} with nonlinearity $(F1)$ or $(F2)$  has no global in time weak solution 
$u(t,x)$. The proof of Proposition~\ref{pr1} is completed.
\end{proof}

Now, we can conclude that if the initial energy satisfy conditions $ (u_0, u_1)>0$,
\be\label{in-1}
E(0)>\max
\left\{\frac{\mathcal{C}(p_1-1)}{2(p_1+1)}\|u_0\|^2, \frac{1}{2}\frac{(u_0,u_1)^2}{\|u_0\|^2}, \frac{\mathcal{C}(p_1-1)}{(1+\mathcal{C})(p_1+1)} (u_0,u_1)
\right \} 
\ee
and
$$
E(0)
<\frac{\mathcal{C}(p_1-1)}{2(p_1+1)}\|u_0\|^2+\frac{1}{2}\frac{(u_0,u_1)^2}{\|u_0\|^2},
$$
then  the result in Theorem~\ref{th7} is completely new. In this way, we obtain a wider class of initial data for which the problem has no global in time solution.

In the following \emph{Example} we 
illustrate Proposition~\ref{pr3} for single nonlinearity \eqref{sn}.

\emph{{Example.}}\label{ex1}
We consider problem \eqref{1}--\eqref{2}, \eqref{sn} with 
initial data of type \eqref{4.2}, where the constant $\sigma$ is set to 1.
Our aim is to choose
constants  $\mu$, $\eta$   and $K$ so  that both conditions \eqref{in-1} and \eqref{2m} hold for $E(0)=K$.
Using the special choice of the initial data \eqref{4.2} with $\sigma=1$, we get
\begin{align*}
&\max\left\{\frac{\mathcal{C}(p_1-1)}{2(p_1+1)}\|u_0\|^2, \frac{1}{2}\frac{(u_0,u_1)^2}{\|u_0\|^2}, \frac{\mathcal{C}(p_1-1)}{(1+\mathcal{C})(p_1+1)} (u_0,u_1)
\right \}= \\
 &\max\left\{ \frac{\mu^2}{2} \frac{\mathcal{C}(p_1-1)}{(p_1+1)} \|w\|^2,  
\frac{\mu^2}{2} \|w\|^2, 
\mu^2 \frac{\mathcal{C}(p_1-1)}{(1+\mathcal{C})(p_1+1)} \|w\|^2\right\}=
\frac{\mu^2}{2} \mathcal{M} \|w\|^2,
\end{align*}
$$
\frac{\mathcal{C}(p_1-1)}{2(p_1+1)}\|u_0\|^2+\frac{1}{2}\frac{(u_0,u_1)^2}{\|u_0\|^2}=
\frac{\mu^2}{2}\mathcal{L}\|w\|^2,
$$
where
$$
\mathcal{M}=\max\left\{\frac{\mathcal{C}(p_1-1)}{(p_1+1)}, 1, \frac{2\mathcal{C}(p_1-1)}{(1+\mathcal{C})(p_1+1)}\right\}, \ \ 
\mathcal{L}=\frac{\mathcal{C}(p_1-1)}{p_1+1}+1.
$$
The requirement that the initial data with $E(0)=K$ simultaneously satisfy both conditions \eqref{in-1} and \eqref{2m}, leads to the following
inequalities:
\be
 \frac{\mu^2}{2} \mathcal{M} \|w\|^2<K=R(\mu)+\frac{\eta^2}{2}\|v\|^2 <\frac{\mu^2}{2}\mathcal{L}\|w\|^2.\label{4.10}
\ee
First, we prove that $\mathcal{M}<\mathcal{L}$. It is enough to show that inequality
\be\label{4.20}
\frac{2\mathcal{C}(p_1-1)}{(1+\mathcal{C})(p_1+1)} <\frac{\mathcal{C}(p_1-1)}{p_1+1}+1
\ee
holds regardless of the value of the constant $\mathcal{C}$.
Indeed, inequality \eqref{4.20} is equivalent to
$$
1>\frac{\mathcal{C} (1-\mathcal{C})(p_1-1)}{(1+\mathcal{C})(p_1+1)},
$$
which is satisfied for every $\mathcal{C}$.
Now, we evaluate $K_0=\frac{\mu_0^2}{2} \frac{(\mathcal{M}+\mathcal{L})}{2} \|w\|^2$.
Then for every constant $K> K_0$, we define $\mu_3=  (2K)^{1/2} \left(\frac{2}{\mathcal{M}+\mathcal{L}}\right)^{1/2}\|w\|^{-1}$,  which guarantees that the inequalities in \eqref{4.10} hold. 
Finally, given that $\mu_3>\mu_0$,  we set the
constant $\eta_3=((2K - R(\mu_3))^{1/2} \|v\|^{-1}$, which gives
the equality $E(0)=E(u^{K}_0,u_1^{K})=K=R(\mu_3)+\frac{\eta_3^2}{2}\|v\|^2$.

In this way we construct initial data $u^{K}_0,u_1^{K}$, defined by \eqref{4.2} with constants $\sigma=1$, $\mu=\mu_3$ and $\eta=\eta_3$, that have sufficiently large positive energy $E(u^{K}_0,u_1^{K})=K>K_0$, satisfy all conditions of Theorem~\ref{th7} and for which the assumptions  \eqref{Straughan-e},  \eqref{Gazzola-e} and \eqref{Xu-e} are not fulfilled.

\section*{Acknowledgments}
The authors
 were partially funded
 by
Grant No BG05M2OP001-1.001-0003, financed by the Science and Education for Smart Growth Operational Program (2014-2020) and co-financed by the European Union through the European structural and Investment funds.
Moreover, the research of the second  author  was partially supported by the Bulgarian Science Fund under 
Grant K$\Pi$-06-H22/2.

\appendix
\section{Existence and uniqueness of local solution}
In this section we will consider the
wave equation \eqref{1}--\eqref{2} with the following nonlinearity 
\be\label{3.1}
f(x,u)= \sum_{i=1}^l d_i(x)|u|^{p_i} + \sum_{i=l+1}^m d_i(x) |u|^{p_i-1}u ,
\ee
where every exponent $p_i$, $i=1,...,m$ satisfies inequalities 
\be\label{3.2}
\begin{split}
&1<p_i<\infty,  \ \ \text{if } n=1,2;\\
&1< p_i <\frac{n}{n-2},  \ \ \text{if }n\geq 3.
\end{split}
\ee
Moreover, we suppose that 
\be\label{3.3}
\begin{split}
&d_i(x)\in\mathrm{C}(\overbar{\Omega}),  ~~~i=1,\ldots,m, \quad
|d_i(x)|\leq D, ~~ \forall x\in \overbar{\Omega},  \ i=1,...,m,\\
&\sum_{i=1}^m|d_i(x)|\not\equiv 0 \ \text{in} \ \Omega.    
\end{split}
\ee
Let us note that  nonlinearity \eqref{3.1} is more general than  $(F1)$  and $(F2)$ because there is no sign restrictions on the coefficients $d_i(x)$, $i=1,...,m$ and the power exponents $p_i$, $i=1,...,m$ are not ordered.

For a given $T>0$ we introduce the Banach space
$$
\mathbb{H}=\{  u\in \mathrm{C}([0,T];\  \mathrm{H}_0^1(\Omega))\cap \mathrm{C}^1([0,T]; \mathrm{L}^2(\Omega))\cap \mathrm{C}^2((0,T_m); \mathrm{H}^{-1}(\Omega)) \}
$$
 equipped  with the norm
$$
\|u\|_{\mathbb{H}}^2=\max_{t\in[0,T]} (\|\nabla u(t)\|^2 +\|u_t(t)\|^2).
$$
In this section we use the short notation $\dot{\theta}(t,x)=\frac{\partial}{\partial t} \theta(t,x)$ for the derivative with respect to $t$.

In order to proof the local existence theorem
we need the following auxiliary lemma.

\begin{lem}\label{lm-th1}
Let $T$ be a fixed positive number. For a given function $u(t,x)\in \mathbb{H}$, $u_0(x)\in \mathrm{H}_0^1(\Omega)$, $u_1(x)\in \mathrm{L}^2(\Omega)$ and $f(x,u)$, defined in \eqref{3.1}--\eqref{3.3}, the problem
\be\label{22}
\begin{split}
&v_{tt} -  \Delta v = f(x,u),  \qquad \quad  t\in(0,T], ~~x \in\Omega,\\
&v(0,x)=u_0(x),~~ v_t(0,x)=u_1(x),  \qquad x \in \Omega, \\
&v(t,x)=0, \qquad \quad t\in [0,T], ~~x\in\partial\Omega
\end{split}
\ee
has a unique weak solution $$v(t,x)\in \mathrm{C}([0,T]; \mathrm{H}_0^1(\Omega))\cap \mathrm{C}^1([0,T]; \mathrm{L}^2(\Omega))\cap \mathrm{C}^2((0,T_m); \mathrm{H}^{-1}(\Omega)).$$
\end{lem}

\begin{proof}
We will use the Galerkin method. If $\{w_j\}$, $j=1,2,...$ is the complete orthonormal system of eigenfunctions to the problem $\Delta w+\lambda w=0$ in  $\mathrm{H}_0^1(\Omega)$, we denote by $\{\lambda_j\}$, $j=1,2,...$ the corresponding eigenvalues, i.e.
\be\label{23}
\Delta w_j+\lambda_j w_j=0, ~~~~ \|w_j\|=\left(\int_\Omega |w_j|^2\, dx\right)^{\frac{1}{2}}=1, \ \ j=1,2,\ldots .
\ee
For the space $\mathcal{W}_k$, generated by $\{ w_1, w_2,...,w_k\}$, and
$$
u_0^k(x)=\sum_{j=1}^k \left(\int_\Omega u_0(x) w_j(x)\, dx\right) w_j(x), ~~~~~~
u_1^k(x)=\sum_{j=1}^k \left(\int_\Omega u_1(x) w_j(x)\, dx\right) w_j(x)
$$
we get $u_0^k\in \mathcal{W}_k$, $\ u_1^k\in \mathcal{W}_k$, $\ u_0^k \to u_0$ in $\mathrm{H}_0^1(\Omega)$ and $\ u_1^k \to u_1$ in $\mathrm{L}^2(\Omega)$
for $k\to\infty$.

We look for functions $\gamma_1^k(t)$,...,$\gamma_k^k(t)\in\mathrm{C}^2([0,T])$ such that
\be\label{24}
v_k(t,x)=\sum_{j=1}^k \gamma_j^k(t) w_j(x)
\ee
satisfies  the problem
\be\label{25}
\begin{split}
&\int_\Omega \left(\ddot{v}_k - \Delta v_k - f(x,u)\right) \mu(x) \, dx=0, \\
&v_k(0) =u_0^k,~~~\dot{v}_k(0)=u^k_1
\end{split}
\ee
for every $\mu(x)\in \mathcal{W}_k$, $t\in[0,T]$. We take $\mu(x)=w_j(x)$ in \eqref{25} and get the identities
$$
\int_\Omega \ddot{v}_k w_j \, dx=\int_\Omega\left(\sum_{s=1}^k \ddot{\gamma}_s^k(t) w_s \right) w_j \, dx=
\ddot{\gamma}_j^k(t)\int_\Omega w_j^2 \, dx =\ddot{\gamma}_j^k(t).
$$
From \eqref{23} it follows that
\begin{align*}
-\int_\Omega \Delta v_k w_j \, dx=&-\int_\Omega \Delta \left(\sum_{s=1}^k \gamma_s^k(t) w_s \right) w_j \, dx=
-\int_\Omega  \left(\sum_{s=1}^k \gamma_s^k(t) \Delta w_s \right) w_j \, dx\\
&=\int_\Omega  \left(\sum_{s=1}^k \gamma_s^k(t) \lambda_s w_s \right) w_j \, dx
=\lambda_j \gamma_j^k(t).
\end{align*}
Hence, for every $j=1,...,k$ the function $\gamma_j^k(t)$ satisfies the Cauchy problem
\be\label{29}
\begin{split}
&\ddot{\gamma}_j^k(t)+\lambda_j \gamma_j^k(t)=h_j(t),\\
&\gamma_j^k(0)=\int_\Omega u_0 w_j \, dx,~~~\dot{\gamma}_j^k(0)=\int_\Omega u_1 w_j \, dx,
\end{split}
\ee
where
$$
h_j(t)=\int_\Omega f(x,u(t,x))w_j(x)\, dx \in \mathrm{C}^0([0,T]).
$$
Since the problem \eqref{29} has a unique solution $\gamma^k_j(t)\in \mathrm{C}^2([0,T])$, then problem \eqref{25} has a unique solution $v_k$ defined in \eqref{24}.
When $\mu=\dot{v}_k$ in \eqref{25}, after integration over $[0,t]$, $0<t\leq T$, we obtain the equality
\be\label{31}
\begin{split}
&\int_\Omega \dot{v}^2_k(t,x) \, dx + \int_\Omega | \nabla {v}_k(t,x)|^2 \, dx \\
&=\|u_1^k\|^2 +\| \nabla u_0^k \|^2 +2 \int_0^t \int_\Omega f(x,u(\tau,x)) \dot{v}_k(\tau,x) \, dx d \tau.
\end{split}
\ee
From  
H\"{o}lder's inequality and the embedding theorem \eqref{Sob},
we get the estimate
\be\label{32}
\begin{split}
&2\int_0^t \int_\Omega f(x,u(\tau,x)) \dot{v}_k(\tau,x) \, dx d\tau \leq 2D 
\int_0^t \int_\Omega \sum_{i=1}^{m} |u|^{p_i} | \dot{v}_k(\tau,x) | \, dx d\tau\\
&\leq \int_0^t \int_\Omega | \dot{v}_k(\tau,x) |^2\, dx d\tau + D^2 \int_0^t \int_\Omega \left(\sum_{i=1}^{m} |u|^{p_i}\right)^2 \, dx d\tau \\
&\leq \int_0^t \int_\Omega | \dot{v}_k(\tau,x) |^2\, dx d\tau + m D^2 \int_0^t \sum_{i=1}^{m} \|u\|^{2p_i}_{2p_i} \,d\tau  \\
&\leq \int_0^t \int_\Omega | \dot{v}_k(\tau,x) |^2\, dx d\tau + m D^2 \int_0^t  \sum_{i=1}^{m} C_{2p_i}^{2p_i} \|\nabla u\|^{2p_i}\,d\tau  \\
&\leq \int_0^t \int_\Omega | \dot{v}_k(\tau,x) |^2\, dx d\tau  + K_1 T,
\end{split}
\ee
where
$$
K_1=m D^2 \ds\sum_{j=1}^{m} C_{2p_i}^{2p_i} \ds\sup_{t\in[0,T]} \|\nabla u\|^{2p_i}.
$$
 Combining \eqref{31} and \eqref{32} we have
\begin{gather*}
\|\dot{v}_k(t)\|^2 +  \| \nabla {v}_k(t)\|^2 \leq 
\|u_1^k\|^2 +\| \nabla u_0^k \|^2 +  \int_0^t \int_\Omega | \dot{v}_k(\tau,x) |^2\, dx d\tau + K_1 T, \\
\leq (\|u_1^k\|^2 +\| \nabla u_0^k \|^2 + K_1 T) +
\int_0^t (\| \dot{v}_k(\tau) \|^2   + \|  \nabla v_k(\tau)\|^2) d\tau.
\end{gather*}

The Gr\"{o}nwall's inequality gives us
$$
\|\dot{v}^2_k(t)\|^2 +  \| \nabla {v}_k(t)\|^2 \leq (\|u_1^k\|^2 +\| \nabla u_0^k \|^2 +K_1 T) \ e^{T}.
$$
Hence,  the  sequence $\{v_k(t,x)\}$, $k=1, 2, \ldots$, is bounded in $\mathrm{L}^\infty([0,T]; \mathrm{H}_0^1(\Omega))$,  
$\{\dot{v}_k(t,x)\},$ $k=1, 2, \ldots$, is bounded in $\mathrm{L}^\infty([0,T]; \mathrm{L}^2(\Omega))$ and
$\{\ddot{v}_k(t,x)\},$ $k=1, 2, \ldots$, is bounded in $\mathrm{L}^\infty([0,T]; \mathrm{H}^{-1}(\Omega))$.
After the limit $k\to\infty$  in \eqref{25}, we get up to a subsequence if it is necessary, that $v_k(t,x)\to v(t,x)$, where $v(t,x)$ is a weak solution to \eqref{22}.

If problem \eqref{22} has two weak solutions $v_1(t,x)$ and $v_2(t,x)$, then function $w=v_1-v_2$ satisfies the problem
\be\label{35}
\begin{split}
&w_{tt}-\Delta w=0, \quad t\in(0,T],  x\in\Omega,\\
&w(0,x)=w_t(0,x)=0, \quad x\in \Omega,\\
&w(t,x)=0, \quad t\in[0,T], ~~x \in \partial\Omega.
\end{split}
\ee
Multiplying \eqref{35} by $w_t$ and integrating over $[0,t]\times \Omega$, we get
$$
\|w_t(t)\|^2 + \| \nabla w(t)\|^2=0,
$$
i.e. $w\equiv 0$ and $v_1(t,x)\equiv v_2(t,x)$ for $(t,x)\in[0,T]\times \Omega$.
Lemma~\ref{lm-th1}  is proved.
\end{proof}

\begin{thm}(Local existence)\label{th1n}
There exists $T>0$ such that problem \eqref{1}--\eqref{2}, \eqref{3.1}--\eqref{3.3}
 admits a unique local weak solution 
 in  $[0,T]$. Moreover, if $T_m= \sup\{T>0: \ u(t,x)~ \text{exists~ on~} [0,T]\times \Omega\}\leq \infty$,  then
\begin{itemize}
\item [(i)] for every $t\in[0,T_m)$ the solution  $u(t,x)$ satisfies the conservation law 

\be\label{consn}
E(0)=E(t),
\ee
where the energy functional $E(t)$ is defined by
\be \label{8n}
\begin{split}
E(t):=E(u(t),u_t(t) )=&\frac{1}{2}\left(
\|u_t(t)\|^2 +\|\nabla u(t)\|^2 \right)
- \int_{\Omega} \int_0^{u(t,x)} f(x,z) \, dz \, d x;
\end{split}
\ee

\item [(ii)]
\be\label{7n}
\text{if} \qquad \limsup_{t\to T_m, t<T_m} \|\nabla u(t)\| <\infty,
\qquad \text{then} \qquad T_m=\infty.
\ee
\end{itemize}
\end{thm}

\begin{proof}

Let $\mathbb{K}^2=2(\|u_1\|^2+\| \nabla u_0\|^2)$. 
We introduce for every $T>0$ the set
$$
\mathbb{R}=\{ u\in\mathbb{H}: \  u(0,x)=u_0(x),~ u_t(0,x)=u_1(x), \  \|u\|_{\mathbb{H}}\leq \mathbb{K} \}.
$$
For a given $u\in \mathbb{R}$ we define the map $\Phi: \mathbb{R} \to \mathbb{H}$ by the rule  $\Phi(u)=v$, where  $v$ is the unique solution to \eqref{22}, given in Lemma~\ref{lm-th1}, with right-hand side $f(x,u)$. We will prove that for a suitable $T>0$ the map $\Phi$ is a contraction map satisfying $\Phi(\mathbb{R})\subset \mathbb{R}$.

Multiplying \eqref{22} by $v_t$, after integration  over $[0,t]\times \Omega$,   we get the identity
\be\label{37}
\|v_t(t)\|^2 +\| \nabla v(t)\|^2=\|u_1\|^2 +\| \nabla u_0\|^2 +
2\int_0^t \int_\Omega f(x,u(\tau,x)) v_t(\tau,x) \, dx d\tau.
\ee
Repeating the proof of the estimate \eqref{32} we obtain the inequality
\be\label{38}
\begin{split}
2\int_0^t \int_\Omega f(x,u(\tau,x)) v_t(\tau,x) \, dx d\tau &\leq
 \int_0^t \|v_t(\tau)\|^2 \, d\tau +  K_1 T\\
&\leq \int_0^t (\|v_t(\tau)\|^2 + \|\nabla v(\tau)\|^2) \, d\tau +  K_2 T
\end{split}
\ee
for every $t\in(0,T]$, where
$$
 K_2=m D^2 \sum_{i=1}^{m} C_{2p_i}^{2p_i} \mathbb{K}^{2p_i} \geq K_1.
$$

 From \eqref{37} and \eqref{38} it follows that
$$
\|v_t(t)\|^2 +\| \nabla v(t)\|^2\leq \frac{1}{2} \mathbb{K}^2 + K_2 T +
\int_0^t (\|v_t(\tau)\|^2 + \|\nabla v(\tau)\|^2) \, d\tau 
$$
and the Gr\"{o}nwall inequality gives us 
$$
\|v_t(t)\|^2 +\| \nabla v(t)\|^2\leq \left(\frac{1}{2} \mathbb{K}^2 +   K_2 T\right) e^T.
$$
Taking the maximum for $t\in[0,T]$ we have 
$$
\|v\|_{\mathbb{H}}^2\leq \left(\frac{1}{2} \mathbb{K}^2 + K_2 T\right) e^T \leq \mathbb{K}^2
$$
for  sufficiently small $T$, i.e. $\Phi(\mathbb{R})\subset \mathbb{R}$.

Now  if $z_1$, $z_2 \in \mathbb{R}$, $v_i=\Phi(z_i)\in\mathbb{R}$, $i=1,2$, then the function $v=v_1-v_2$ satisfies the problem
\be\label{41}
\begin{split}
&v_{tt}-\Delta v=f(x,z_1) - f(x,z_2), \ \ t\in(0,T], \ x\in\Omega,\\
&v(0,x)=v_t(0,x)=0~~~~~ x\in \Omega,\\
&v(t,x)=0, ~~~~~ t\in[0,T], ~~x \in \partial\Omega.
\end{split}
\ee
Multiplying \eqref{41} by $v_t$ and integrating over $[0,t]\times\Omega$, we obtain the following chain of inequalities
\be\label{42}
\begin{split}
&\| v_t(t)\|^2 + \|\nabla v(t)\|^2  \leq 2\int_0^t \int_\Omega (f(x,z_1)-f(x,z_2)) v_t(\tau,x) \, dx d\tau \\
&\leq 2\int_0^t \int_\Omega \sum_{i=1}^m \left( p_i |d_i(x)|  \int_0^1|(1-s)z_1+s z_2|^{p_i-1}\, ds\right) |z_1-z_2|\, |v_t(\tau,x)| \, dx d\tau \\
& \leq 2 p D \int_0^t \int_\Omega \sum_{i=1}^m (|z_1|+|z_2|)^{p_i-1} \ |z_1-z_2| \  |v_t(\tau,x)| \ dx d\tau,  \ p=\max_{i\in [1,m]}\{p_i\}. \\
\end{split}
\ee
For $n\geq 2$,
applying the generalized H\"{o}lder's inequality, we get from \eqref{42}
\be\label{42n}
\begin{split}
&\| v_t(t)\|^2 + \|\nabla v(t)\|^2  
\leq 2 p D \int_0^t \sum_{i=1}^m  \left\|(|z_1|+|z_2|)^{p_i-1} \right \|_{\frac{2n}{1-\delta_i}} \ \|z_1-z_2\|_{\frac{2n}{n-1+\delta_i}} \ \|v_t(\tau)\| \ d\tau,
\end{split}
\ee
where
$$\delta_i=\begin{cases}
\max\{1-4(p_i-1),0\} &\text{if} \ n=2; \\
 1-(p_i-1)(n-2) &\text{if} \ n\geq 3. \\
\end{cases}
$$ 
Note that $\delta_i\in (-1,1)$ under conditions  \eqref{3.2}.

From the trivial inequality $(a+b)^p\leq 2^p (a^p +b^p)$ for $p>0$, $a\geq 0$, $b\geq 0$, we have 
\begin{align*}
 &\left(\int_\Omega\left(|z_1|+|z_2|\right)^{\frac{2n(p_i-1)}{1-\delta_i}} \ dx\right)^{\frac{1-\delta_i}{2n}} \leq 
2^{p_i-1} \left(\int_\Omega \left(|z_1|^{\frac{2n(p_i-1)}{1-\delta_i}}+|z_2|^{\frac{2n(p_i-1)}{1-\delta_i}} \right)\ dx \right)^{\frac{1-\delta_i}{2n}}\\
&
 \leq 2^{(p_i-1)+\frac{(1-\delta_i)}{2n}} \left\{
\left(\int_\Omega |z_1|^{\frac{2n(p_i-1)}{1-\delta_i}} \, dx \right)^{\frac{1-\delta_i}{2n}}\ +\left(\int_\Omega |z_2|^{\frac{2n(p_i-1)}{1-\delta_i}} \, dx \right)^{\frac{1-\delta_i}{2n}} 
\right\}\\
&\leq 2^{p_i-1+\frac{1}{2n}} \left( \|z_1\|_{\frac{2n(p_i-1)}{1-\delta_i}}^{p_i-1} +\|z_2\|_{\frac{2n(p_i-1)}{1-\delta_i}}^{p_i-1}\right)\leq 2^{p_i} C_{\frac{2n(p_i-1)}{1-\delta_i}}^{p_i-1}
\left( \|\nabla z_1\|^{p_i-1} +\|\nabla z_2\|^{p_i-1}\right).
\end{align*}
In the above estimates the Sobolev inequalities \eqref{Sob} are  applicable  since
$$
1\leq \frac{2n(p_i-1)}{1-\delta_i} \quad \text {if} \ n=2; \qquad 1\leq \frac{2n(p_i-1)}{1-\delta_i}=\frac{2n}{n-2} \quad \text {if} \ n\geq 3.
$$

From $1<\frac{2n}{n}<\frac{2n}{n-1-\delta_i}<\frac{2n}{n-2}$, applying  \eqref{Sob} for  $\|z_1-z_2\|_{\frac{2n}{n-1+\delta_i}}$,  inequality \eqref{42n} becomes
\be\label{42nn}
\begin{split}
&\| v_t(t)\|^2 + \|\nabla v(t)\|^2 \\
&\leq p D\sum_{i=1}^m 2^{p_i+1} \, C_{{\frac{2n(p_i-1)}{1-\delta_i}}}^{p_i-1} \, C_{\frac{2n}{n-1-\delta_i}} 
\int_0^t  \left( \|\nabla z_1\|^{p_i-1} +\|\nabla z_2\|^{p_i-1}\right)
 \ \| \nabla(z_1-z_2)\| \ \|v_t(\tau)\| \ d\tau \\
&\leq
\int_0^t \|v_t(\tau)\|^2 \, d\tau +p^2 D^2 \sum_{i=1}^m 2^{2p_i+1}  \, C_{{\frac{2n(p_i-1)}{1-\delta_i}}}^{2(p_i-1)} \, C_{\frac{2n}{n-1-\delta_i}}^2
\int_0^t  \left( \|z_1\|_{\mathbb{H}}^{p_i-1} +\|z_2\|_{\mathbb{H}}^{p_i-1}\right)^2
 \ \| z_1-z_2\|_{\mathbb{H}}^2 \ d\tau \\
&\leq
\int_0^t (\|v_t(\tau)\|^2 +\| \nabla v(\tau)\|^2) \, d\tau +2^{2p+2} p^2 D^2  \, \sum_{i=1}^m C_{{\frac{2n(p_i-1)}{1-\delta_i}}}^{2(p_i-1)} \, C_{\frac{2n}{n-1-\delta_i}}^2 \mathbb{K}^{2(p_i-1)}
\ \| z_1-z_2\|_{\mathbb{H}}^2 \, T. \\
\end{split}
\ee

For $n=1$, from \eqref{42} and \eqref{Sob}, we get
\be\label{421}
\begin{split}
&\| v_t(t)\|^2 + \|\nabla v(t)\|^2  \leq 
2 p D \int_0^t  \int_{\Omega} \sum_{i=1}^m \left(\|z_1\|_{\infty}+\|z_2\|_{\infty}\right)^{p_i-1} \ |z_1-z_2| \  |v_t(\tau,x)| \ dx d\tau \\ 
&\leq 2 p D  \int_0^t \int_{\Omega}\sum_{i=1}^m 2^{p_j} C_{\infty}^{p_i-1}  \mathbb{K}^{p_i-1} \ |z_1-z_2| \  |v_t(\tau,x)| \ dx d\tau \\ 
& \leq 2^{p+1}  p D \sum_{i=1}^m C_{\infty}^{p_i-1}\mathbb{K}^{p_i-1} \int_0^t \|z_1-z_2\|_{\infty} \|v_t(\tau)\| \, d\tau \\
&\leq
\int_0^t \|v_t(\tau)\|^2 \, d\tau + 2^{2p+1} p^2  D^2\sum_{i=1}^m C_{\infty}^{2p_i-2}  \mathbb{K}^{2(p_i-1)} C_{\infty}^2\|\nabla (z_1 - z_2)\|^2  T \\
&
\leq 
\int_0^t (\|v_t(\tau)\|^2 +\| \nabla v(\tau)\|^2) \, d\tau +  2^{2p+1}  p^2 D^2 \sum_{i=1}^m C_{\infty}^{2p_j} \mathbb{K}^{2(p_i-1)} \|z_1 - z_2\|^2_{\mathbb{H}}\, T.
\end{split}
\ee
Estimates \eqref{42nn} (case $n\geq 2$) and \eqref{421} (case $n=1$) can be combined as
$$
\| v_t(t)\|^2 + \|\nabla v(t)\|^2  \leq 
\int_0^t (\|v_t(\tau)\|^2 +\| \nabla v(\tau)\|^2) \, d\tau + K_3  \| z_1-z_2\|_{\mathbb{H}}^2 \, T,
$$
where
$$
K_3=\begin{cases}
2^{2p+1}  p^2 D^2 \ds\sum_{i=1}^m C_{\infty}^{2p_i} \mathbb{K}^{2(p_i-1)} & \text{for}~~n=1;\\[15pt]
2^{2p+2} p^2 D^2 \ds \sum_{i=1}^m C_{{\frac{2n(p_i-1)}{1-\delta_i}}}^{2(p_i-1)} \, C_{\frac{2n}{n-1-\delta_i}}^2 \mathbb{K}^{2(p_i-1)} & \text{for}~~n\geq2.
\end{cases}
$$
Using Gr\"{o}nwall's inequality we get
$$
\|v_1-v_2\|^2_{\mathbb{H}}=\max_{t\in[0,T]}\| v_t(t)\|^2 + \|\nabla v(t)\|^2 \leq
K_3 T e^T \|z_1-z_2\|^2_{\mathbb{H}}, 
$$
i.e.
$$
\|\Phi(z_1) -\Phi(z_2)\|^2_{\mathbb{H}}
\leq \xi \|z_1-z_2\|^2_{\mathbb{H}},
$$
where
$$
\xi=K_3 T e^T <1~~~ \text{for sufficiently small} \ T.
$$
Hence, for sufficiently small T, the map $\Phi$ is a contraction and from the contraction mapping principle, there exists a unique weak solution to \eqref{1}--\eqref{2}, \eqref{3.1}--\eqref{3.3} for $t\in[0,T]$.

Let $[0,T_m)$ be the maximal  existence time interval of the weak solution to \eqref{1}--\eqref{2}, \eqref{3.1}--\eqref{3.3}.
Differentiating \eqref{8n} with respect to $t$ we get from \eqref{1} the identity
$$
E'(t)=\int_\Omega u_t(u_{tt}-\Delta u - f(x,u)) \, dx=0~~~~\text{for}~~t\in[0,T_m),
$$
which proves the conversation law of the energy \eqref{consn}.

If \eqref{7n} holds, then from the embedding theorem  \eqref{Sob} it follows that
$$
\limsup_{t\to T_m, t<T_m} \|u(t)\|_p <\infty ~~~\text{for}~~p~~\text{satisfying \eqref{3.2}}
$$
and from conservation law  \eqref{8n} we come to
$$
 \limsup_{t\to T_m, t<T_m} \| u_t(t)\| <\infty.
$$ 
Applying  the same arguments  as in the proof above, we conclude that for some $\delta >0$ problem \eqref{1}--\eqref{2}, \eqref{3.1}--\eqref{3.3} has a weak solution in the interval $[0,T_m+\delta)$. This contradicts that $T_m$ is  the maximal existence time interval.
The proof of Theorem~\ref{th1n} is completed. 
\end{proof}


\begin{thebibliography}{50}

\bibitem{Ball1}
J. Ball,
Remarks on blow-up and nonexistence theorems for nonlinear evolution equations,
Q. J. Math.
28 (1977) 473--486. 

\bibitem{Ball2}
J. Ball,
Finite time blow-up in nonlinear problems, 
in: Proceedings, Symposium on Nonlinear Evolution Equations, University of Wisconsin, Madison,  1977, 
Academic Press, New York-London, 1978,  89--205. 


\bibitem{Tsutsumi}
M. Tsutsumi,
On solutions of semilinear differential equations in a Hubert space, 
Math. Japon. 17  (1972) 173--193.



\bibitem{Glassey}
R.T. Glassey, 
Blow-up theorems for nonlinear wave equations,
Math. Z. 132 (1973) 183--203.



\bibitem{Levine-1}
H.A. Levine,
Instability and nonexistence of global solutions to nonlinear wave equations of the form
$Pu_{tt}=-Au+F(u)$,  
Trans. Amer. Math. Soc. 192 (1974) 1--21.

\bibitem{Levine-2}
H.A. Levine,
Some additional remarks on the nonexistence of global solutions to nonlinear wave equations, 
SIAM J. Math. Anal. 5 (1974) 138--146.

\bibitem{Levine-3}
H.A. Levine;
A note on a nonexistence theorem for nonlinear wave equations, 
SIAM J. Math. Anal. 5  (1974) 644--648.


\bibitem{Straughan}
B. Straughan,
Further global nonexistence theorems for abstract nonlinear wave equations,
Proc. Amer. Math. Soc. 48 (1975)  381--390. 


\bibitem{K-L}
V. K. Kalantarov, O.A. Ladyzhenskaya,
The occurrence of collapse for quasilinear equations of parabolic and hyperbolic types,
 J. Soviet Math. 10  (1978)   53--70.


\bibitem{Korpusov-1}
M.O. Korpusov,
Blowup of a positive-energy solution of model wave
equations in nonlinear dynamics, 
Theoret. and Math. Phys. 171(1) (2012) 421--434.


\bibitem{EJDE}
M. Dimova, N. Kolkovska, N. Kutev,
Blow up of solutions to ordinary differential equations arising in nonlinear dispersive problems, 
Electron. J. Differential Equations 2018 (2018) No. 68 1--16. 



\bibitem{Korpusov}
M.O. Korpusov, 
Non-existence of global solutions to generalized dissipative Klein-Gordon equations with positive energy,
Electron. J. Differential Equations 2012 (2012) No. 119 1--10.



\bibitem{Kalantarov}
B. A. Bilgin, V. K. Kalantarov,
Non-existence of global solutions to nonlinear
wave equations with positive initial energy,
Commun. Pure Appl. Anal. 17(3) (2018) 987--999.


\bibitem{EJDE-2022}
S. Chen, R. Xu, C. Yang,
Improved blowup time estimates for fourth-order damped wave equations with strain term and arbitrary positive initial energy,
Electron. J. Differential Equations 2022 (2022), No. 70, 1--13. 



\bibitem{Avila-inequality}
 J. Esquivel-Avila,
A differential inequality and the blow-up of its solutions,
Appl. Math. E-Notes 22 (2022) 178--183.


\bibitem{S}
D.H. Sattinger, 
On global solution of nonlinear hyperbolic equations, 
Arch. Ration. Mech. Anal. 30 (1968) 148--172. 

\bibitem{P-S}
L.E. Payne, D.H. Sattinger, 
Saddle points and instability of nonlinear hyperbolic equations,
Israel J. Math. 22(3-4) (1975) 273--303.


\bibitem{Liu}
Yacheng Liu, Junsheng Zhao,
On potential wells and applications to semilinear hyperbolic equations and parabolic equations,
Nonlinear Anal. 64 (12) (2006)  2665--2687.


\bibitem{Avila1} 
 J. Esquivel-Avila, 
A characterization of global and nonglobal
solutions of nonlinear wave and Kirchhoff equation,
Nonlinear Anal. 52 (2003)  1111--1127.


\bibitem{Avila2} 
 J. Esquivel-Avila, 
Blow up and asymptotic behavior in a nondissipative nonlinear wave equation,  
Appl. Anal. 
93 (2014) 1963--1978.


\bibitem{Xu-Quaterly}
R. Xu,
 Initial boundary value problem for semilinear hyperbolic equations and parabolic
equations with critical initial data,
Quart. Appl. Math. 68 (2010) 459--468.



\bibitem{Xu-several}
Y. Liu, R. Xu,
Wave equations and reaction–diffusion equations with several nonlinear source terms of different sign,
Discrete Contin. Dyn. Syst. Ser. B 7 (2007) 171--189.


\bibitem{Xu-EJDE}
R. Xu, Y. Chen, Y. Yang, S. Chen, J. Shen, T. Yu,  Z. Xu, 
Global well-posedness of semilinear hyperbolic equations, parabolic equations and Schrodinger equations,
Electron. J. Differential Equations 2018 (2018), No. 55, 1--52. 



%
\bibitem{Dai}
X. Dai,
Hyperbolic and parabolic equations with several opposite-sign source at critical initial energy level in heat and vibrating systems,
Nonlinear Anal. 195 (2020)  111752.
%




\bibitem{Xu-log}
W. Lian, Md S. Ahmed, R. Xu,
 Global existence and blow up of solution for semilinear hyperbolic equation with logarithmic nonlinearity, 
Nonlinear Anal. 184 (2019) 239--257.




\bibitem{Antontsev}
S. Antontsev, S. Shmarev,
Evolution PDEs with nonstandard growth conditions. Existence, uniqueness, localization, blow-up, Atlantis Press, Paris, 2015.



\bibitem{Piskin}
S. Antonsev, J. Ferreira, E. Piskin,
Existence and blow up of solutions for a strongly damped petrovsky equation with
variable-exponent nonlinearities,
Electron. J. Differential Equations  2021 (2021), No. 6, 1--18.


\bibitem{Kalantarova}
J. Kalantarova,
Blow up of solutions to semilinear non-autonomous wave equations under robin boundary conditions,
Math. Notes 106  (2019) 364--371.


\bibitem{Georgiev}
V. Georgiev and G. Todorova, Existence of a solution of the wave equation with nonlinear damping and source terms,
J. Differential  Equations 109 (1994) 295--308.



\bibitem{Gazzola}
F. Gazzola, M. Squassina, 
Global solutions and finite time blow up for damped semilinear wave equations, 
Ann. I. H. Poincar\'{e} -- AN 23 (2006) 185--207.


\bibitem{Avila3}
 J. Esquivel-Avila,
The dynamics of a nonlinear wave equation,
J. Math. Anal. Appl.
279 (2003)  135--150.





\bibitem{Vitillaro-1}
E. Vitillaro,
A potential well theory for the wave equation with nonlinear source and
boundary damping terms,
Glasg. Math. J. 44 (2002)  375--395.


\bibitem{Vitillaro-2}
E. Vitillaro,
On the wave equation with hyperbolic dynamical boundary conditions, interior and boundary damping and supercritical sources,
J. Differential Equations 265 (2018) 4873--4941.


\bibitem{Xu-ANZIAM}
R. Xu, Y. Yang, S. Chen, J. Su, J. Shen, S. Huang, 
Nonlinear wave equations and reaction-diffusion equations with several nonlinear source terms of different signs at high energy level,
ANZIAM J. 54  (2013) 153--170. 



\bibitem{Xu-CPAA}
Y. Yang, R. Xu,
 Nonlinear wave equation with both strongly and weakly damped terms: Supercritical initial energy finite time blow up, 
Commun. Pure Appl. Anal. 18 (2019) 1351--1358.



\bibitem{Avila-wave}
 J. Esquivel-Avila,
Nonexistence of global solutions of abstract wave
equations with high energies,
J. Inequal. Appl.  2017 (2017), No. 268.








\bibitem{NTADES2022} 
N. Kutev, M. Dimova, N. Kolkovska,
Finite time blow up of the solutions to nonlinear wave equations with sign-changing nonlinearities, 
in: Slavova, A. (Eds) 
New Trends in the Applications of Differential Equations in Sciences, 
Springer Proc. Math. Stat.  412 (2023)  83--93.




\bibitem{ERA} 
M. Dimova, N. Kolkovska, N. Kutev, 
Global behavior of the solutions to nonlinear Klein--Gordon equation with critical initial energy, 
Electron. Res. Arch. 28(2) (2020) 671--689.



\end{thebibliography}
\end{document}